\documentclass[10pt,leqno]{article}

\usepackage{amsmath,amssymb,amsthm,mathrsfs,dsfont}

\usepackage[margin=3cm]{geometry} %宽版时用，配合双倍行距

\usepackage{titlesec,hyperref}

\usepackage{color}

\usepackage{fancyhdr}
\titleformat{\subsection}{\it}{\thesubsection.\enspace}{1pt}{}

\newtheorem{theo}{Theorem}[section]
\newtheorem{lemm}[theo]{Lemma}

\newtheorem{prop}[theo]{Proposition}
\newtheorem{rema}[theo]{Remark}
\numberwithin{equation}{section}

\allowdisplaybreaks %允许公式分页显示

\newcommand\lm{{\lesssim}}
\begin{document}
\title{Global strong solutions and large time behavior to the compressible co-rotation FENE dumbbell model of polymeric flows near equilibrium
\hspace{-4mm}
}

\author{ Zhaonan $\mbox{Luo}^1$ \footnote{email: 1411919168@qq.com},\quad
Wei $\mbox{Luo}^1$\footnote{E-mail:  luowei23@mail2.sysu.edu.cn} \quad and\quad
 Zhaoyang $\mbox{Yin}^{1,2}$\footnote{E-mail: mcsyzy@mail.sysu.edu.cn}\\
 $^1\mbox{Department}$ of Mathematics,
Sun Yat-sen University, Guangzhou 510275, China\\
$^2\mbox{Faculty}$ of Information Technology,\\ Macau University of Science and Technology, Macau, China}

\date{}
\maketitle
\hrule

\begin{abstract}
In this paper, we mainly study global well-posedness and optimal decay rate for the strong solutions of the compressible co-rotation finite extensible nonlinear elastic (FENE) dumbbell model. This model is a coupling of the isentropic compressible Navier-Stokes equations with a nonlinear Fokker-Planck equation. We first prove that the FENE dumbbell model admits a unique global strong solution provided the initial data are close to equilibrium state for $d\geq 2$. Moreover, for $d\geq3$, we show that optimal decay rates of global strong solutions by the linear spectral theory and a more precise Hardy type inequality.  \\
\vspace*{5pt}

\vspace*{5pt}
\end{abstract}

\vspace*{10pt}

%\phantomsection
%\addcontentsline{toc}{section}{\contentsname}
%添加目录到书签
\tableofcontents

\section{Introduction}
   In this paper we consider the compressible finite extensible nonlinear elastic (FENE) dumbbell model \cite{Bird1977,Doi1988}:
\begin{align}\label{eq0}
\left\{
\begin{array}{ll}
\varrho_t+{\rm div}(\varrho u)=0 , \\[1ex]
(\varrho u)_t+{\rm div}(\varrho u\otimes u)-{\rm div}\Sigma{(u)}+\frac 1 {Ma^2} \nabla{P(\varrho)}=\frac 1 {De} \frac {\kappa} {r} {\rm div}~\tau, \\[1ex]
\psi_t+u\cdot\nabla\psi={\rm div}_{R}[- \sigma(u)\cdot{R}\psi+\frac {\sigma} {De}\nabla_{R}\psi+\frac {1} {De\cdot r}\nabla_{R}\mathcal{U}\psi],  \\[1ex]
\tau_{ij}=\int_{B}(R_{i}\nabla_{Rj}\mathcal{U})\psi dR, \\[1ex]
\varrho|_{t=0}=\varrho_0,~~u|_{t=0}=u_0,~~\psi|_{t=0}=\psi_0, \\[1ex]
(\sigma\nabla_{R}\psi+\frac {1} {r}\nabla_{R}\mathcal{U}\psi)\cdot{n}=0 ~~~~ \text{on} ~~~~ \partial B(0,R_{0}) .\\[1ex]
\end{array}
\right.
\end{align}
In \eqref{eq0}, $\varrho(t,x)$ is the density of the solvent, $u(t,x)$ stands for the velocity of the polymeric liquid and $\psi(t,x,R)$ denotes the distribution function for the internal configuration. Here the polymer elongation $R$ is bounded in ball $ B=B(0,R_{0})$ which means that the extensibility of the polymers is finite and $x\in\mathbb{R}^d$. The notation $\Sigma{(u)}=\mu(\nabla u+\nabla^{T} u)+\mu'{\rm div}~u\cdot Id$ represents the stress tensor, with the viscosity coefficients $\mu$ and $\mu'$ satisfying $\mu>0$ and $2\mu+\mu'>0$. $\sigma$ is a constant satisfied the relation $\sigma=k_BT_a$, where $k_B$ is the Boltzmann constant, $T_a$ stands for the absolute temperature.  Moreover, $\kappa > 0$ denotes the ratio between kinetic and elastic energy and $r>0$ is related to the linear damping mechanism in dynamics of the microscopic variable $R$.
The parameter $De$ is the Deborah number, which stands for the ratio of the time scales for elastic stress relaxation. It measures the fluidity of the system. The smaller the Deborah number is, the system behaves more like a Newtonian fluid. Furthermore, the Mach number $Ma$ denotes the ratio between the fluid velocity and the sound speed, so it characterizes the compressibility of the system.
The pressure satisfies the so-called $\gamma$-law: $P(\varrho)=a\varrho^\gamma$ with $\gamma\geq1, a>0$. $\tau$ is an additional stress tensor.
For the compressible FENE dumbbell model, the potential $\mathcal{U}(R)=-k\log(1-(\frac{|R|}{|R_{0}|})^{2})$ for some constant $k>0$. $\sigma(u)$ is the drag term. In the co-rotation case, $\sigma(u)=\frac {\nabla u-(\nabla u)^{T}} {2}$. In the general case, $\sigma(u)=\nabla u$.

 This model describes the system coupling fluids and polymers. The system is of great interest in many branches of physics, chemistry, and biology, see \cite{Bird1977,Doi1988}. In this model, a polymer is idealized as an "elastic dumbbell" consisting of two "beads" joined by a spring that can be modeled by a vector $R$. The polymer particles are studied by a probability function $\psi(t,x,R)$ satisfying that $\int_{B}
\psi(t,x,R)dR =1$, which stands for the distribution of particles elongation vector $R\in B$. At the level of liquid, the system couples the compressible Navier-Stokes equations for the fluid velocity with a Fokker-Planck equation describing the evolution of the polymer density. This is a micro-macro model (For more details, one can refer to  $\cite{Masmoudi2008}$ and $\cite{Masmoudi2013}$).

In this paper we will take $a,~\sigma,~\kappa,~r,~De,~Ma$ and $R_{0}=1$.
Notice that $(\varrho,u,\psi)$ with $\varrho=1$, $u=0$ and $$\psi_{\infty}(R)=\frac{e^{-\mathcal{U}(R)}}{\int_{B}e^{-\mathcal{U}(R)}dR}=\frac{(1-|R|^2)^k}{\int_{B}(1-|R|^2)^kdR},$$
is a trivial solution of \eqref{eq0}. Then we study the perturbations near the global equilibrium:
\begin{align*}
\rho=\varrho-1,~~u=u,~~g=\frac {\psi-\psi_\infty} {\psi_\infty}.
\end{align*}
By a simple calculation, we get
\begin{align}\label{co}
&{\rm div}_R([(\nabla u-(\nabla u)^T]R\psi_\infty)
=\sum_{i,j}\partial_{R_i}[(\partial_iu^j-\partial_ju^i)R_j\psi_\infty]
\\ \notag
&=\sum_{i,j}(\partial_iu^j-\partial_ju^i)\delta_{ij}\psi_\infty+\sum_{i,j}\frac{2k(\partial_iu^j-\partial_ju^i)R_jR_i(1-|R|^2)^{k-1}}{\int_{B}(1-|R|^2)^kdR}
=0.
\end{align}
Then we can rewrite \eqref{eq0} for the co-rotation case ($\sigma(u)=\frac {\nabla u-(\nabla u)^{T}} {2}$)  as the following system:
\begin{align}\label{eq1}
\left\{
\begin{array}{ll}
\rho_t+{\rm div}~u(1+\rho)=-u\cdot\nabla\rho , \\[1ex]
u_t-\frac 1 {1+\rho} {\rm div}\Sigma{(u)}+\frac {P'(1+\rho)} {1+\rho} \nabla\rho=-u\cdot\nabla u+\frac 1 {1+\rho} {\rm div}~\tau, \\[1ex]
g_t+\mathcal{L}g=-u\cdot\nabla g-\frac 1 {\psi_\infty}\nabla_R\cdot(\sigma(u)Rg\psi_\infty),  \\[1ex]
\tau_{ij}(g)=\int_{B}(R_{i}\nabla_{Rj}\mathcal{U})g\psi_\infty dR, \\[1ex]
\rho|_{t=0}=\rho_0,~~u|_{t=0}=u_0,~~g|_{t=0}=g_0, \\[1ex]
\psi_\infty\nabla_{R}g\cdot{n}=0 ~~~~ \text{on} ~~~~ \partial B(0,1) ,\\[1ex]
\end{array}
\right.
\end{align}
where $\mathcal{L}g=-\frac 1 {\psi_\infty}\nabla_R\cdot(\psi_\infty\nabla_{R}g)$.

{\bf Remark.} As in the reference \cite{Masmoudi2013}, one can deduce that $\psi=0$ on $\partial B(0,1)$.

\subsection{Short reviews for the incompressible FENE dumbbell model}
We first review some mathematical results about the incompressible FENE dumbbell model. M. Renardy \cite{Renardy} established the local well-posedness in Sobolev spaces with potential $\mathcal{U}(R)=(1-|R|^2)^{1-\sigma}$ for $\sigma>1$. Later, B. Jourdain, T. Leli\`{e}vre, and
C. Le Bris \cite{Jourdain} proved local existence of a stochastic differential equation with potential $\mathcal{U}(R)=-k\log(1-|R|^{2})$ in the case $k>3$ for a Couette flow. H. Zhang and P. Zhang \cite{Zhang-H} proved local well-posedness for the FENE equation with $d=3$ in weighted Sobolev spaces. For the co-rotation case, F. Lin, P. Zhang, and Z. Zhang \cite{F.Lin} obtained a global existence results with $d=2$ and $k > 6$. If the initial data is perturbation around equilibrium, N. Masmoudi \cite{Masmoudi2008} proved global well-posedness for for the general case and $k>0$. In the co-rotation case with $d=2$, he \cite{Masmoudi2008} obtained a global result $k>0$ without any small conditions. In the co-rotation case, A. V. Busuioc, I. S. Ciuperca, D. Iftimie and L. I. Palade \cite{Busuioc} obtained a global existence result with only the small condition on $\psi_0$. The global existence of weak solutions in $L^2$ for the general case and  was proved recently by N. Masmoudi \cite{Masmoudi2013} under some entropy conditions. In this paper, he point out that global existence of strong solutions for the general case is an open problem.

M. Schonbek \cite{Schonbek} studied the $L^2$ decay of the weak solutions for the co-rotation
FENE dumbbell model, and obtained the
decay rate $(1+t)^{-\frac{d}{4}+\frac{1}{2}}$, $d\geq 2$ with $u_0\in L^1$.
Moreover, she conjectured that the sharp decay rate should be $(1+t)^{-\frac{d}{4}}$,~$d\geq 2$.
However, she failed to get it because she could not use the bootstrap argument as in \cite{Schonbek1985} due to the
additional stress tensor.  Recently, W. Luo and Z. Yin \cite{Luo-Yin} improved Schonbek's result
and showed that the decay rate is $(1+t)^{-\frac{d}{4}}$ with $d\geq 3$ and $\ln^{-l}(1+t)$ with $d=2$ for any $l\in\mathbb{N^+}$.
This result shows that M. Schonbek's conjecture is true when $d \geq3$. More recently, W. Luo and Z. Yin \cite{Luo-Yin2} improved the decay rate to $(1+t)^{-\frac{d}{4}}$ with $d=2$.

\subsection{Short reviews for the CNS equations}
Taking $\psi\equiv 0$, the system \eqref{eq0} reduce to the compressible Navier-Stokes (CNS) equations. In order to study about the \eqref{eq0}, we have to cite some reference about the CNS equations. The first local existence and uniqueness results were obtained by J. Nash \cite{miaocompressns23} for smooth initial data without vacuum. Later on, A. Matsumura and T. Nishida \cite{Matsumura} proved the global well-posedness for smooth data close to equilibrium in $H^3\times H^2$. In \cite{miaocompressns18}, A. V. Kazhikhov and V. V. Shelukhin established the first global existence result with large data in one dimensional space under some suitable condition on $\mu$ and $\lambda$. If $\mu$ is constant and $\lambda(\rho)=b\rho^\beta$, X. Huang and J. Li \cite{Huang} obtained a global existence and uniqueness result for large initial data in two dimensional space(See also \cite{miaocompressns26}). In \cite{miaocompressns25}, X. Huang, J. Li, and Z. Xin proved the global well-posedness with vacuum. The blow-up phenomenons were studied by Z. Xin et al
in \cite{Xin98,miaocompressns28,miaocompressns27}. Concerning the global existence of weak solutions for the large initial data, we may refer to \cite{miaocompressns2,miaocompressns3,miaocompressns21,Vasseur}.

To catch the scaling invariance property of the CNS equations. R. Danchin introduced the critical spaces in his series papers \cite{miaocompressns11,miaocompressns12,miaocompressns14,miaocompressns15,miaocompressns155} and obtained several important existence and uniqueness results. Recently, Q. Chen, C. Miao and Z. Zhang \cite{miao} proved the local existence and uniqueness in critical homogeneous Besov spaces. The ill-posedness result was obtained in \cite{miaoill-posedness}. In \cite{miaocompressns24}, L. He, J. Huang and C. Wang proved the global stability with $d=3$ i.e. for any perturbed solutions will remain close to the reference solutions if initially they are close to another one.

The large time behaviour was proved by A. Matsumura and T. Nishida in \cite{Matsumura}. H. Li and T. Zhang \cite{Li2011Large} obtained the optimal time decay rate for the CNS equations by spectrum analysis in Sobolev spaces. Recently, J. Xu \cite{Xu2019} studied about the large time behaviour in the critical Besov space and obtain the optimal time decay rate.

\subsection{Main results}

The well-posedness in $H^3$ with the Hooke type potential and $d=3$ for the system \eqref{eq0} was established by N. Jiang, Y. Liu and T. Zhang \cite{2017Global}. They proved the global well-posedness for \eqref{eq0} if the initial data is close to the equilibrium. In \cite{2017Global}, the authors assume that $R\in\mathbb{R}^3$ which means that polymer elongation may be infinite. Actually, the polymer elongation $R$ is usually bounded.

Recently, N. Masmoudi \cite{2016Equations} was concerned with long time behavior for polymeric models. To our best knowledge, well-posedness and large time behaviour for the system \eqref{eq0} with finite polymer elongation has not been studied yet. In this paper, we establish global well-posedness result for the compressible co-rotation FENE equation \eqref{eq0} if the initial data is close to the equilibrium. The key point is to prove a priori estimate which is global in time for \eqref{eq1} with small data. Compared to \eqref{eq0} with the Hooke potential $\mathcal{U}(R)=\frac{1}{2}|R|^2$, the main difficult for the FENE system is to control the stress tensor $\tau$. That is why many researchers establish different Hardy type inequalities. Taking advantage of dissipative structure of \eqref{eq1} and the interpolation method, we obtain the lower order energy estimates. By virtue of the Hardy type inequality \cite{Masmoudi2008} and cancellation relation between the compressible Navier-Stokes equations and Fokker-Planck equation, we get the higher order derivatives estimates for \eqref{eq1}. Combining the lower and higher order estimates, we deduce a closed estimate which is global in time. The result of the global existence of
strong solution of the current manuscript is an extension of \cite{Masmoudi2008} to the compressible fluid.

Moreover, we study about large time behaviour and obtain optimal decay rate for the velocity in $L^2$. The proof is based on the linear spectral theory and $L^2$ energy estimate method with the conditions $(\rho_0,u_0)\in L^1\times L^1$. The main difficulty is to estimate the additional linear term ${\rm div}~\tau$. To get $L^2$ decay rate, we prove a more precise Hardy type inequality to control the extra stress tensor $\tau$ by improving the method in \cite{Masmoudi2013}. Since ${\rm div}~\tau$ is a linear term, it follows that we have to estimate the $L^1$-norm for $\tau$. For this purpose, we add the condition $g_0\in L^1(\mathcal{L}^p)$ with $(p-1)k>1$ and $2\leq p$. In the incompressible case, one can prove that $\|\tau\|_{L^1}\leq C$ for any $t$ (See \cite{Masmoudi2008}). However, in the compressible case, we only obtain that $\|\tau \|_{L^1}\leq Ce^{\sqrt{t}}$, which is exponentially growth in time. This is too bad for study about optimal decay rate.  Fortunately, in the co-rocation case, we see that the $L^2(\mathcal{L}^2)$ norm $g$ is exponentially decay in time. Although we can not obtain exponential decay estimate of $\tau$ through the new Hardy estimate, we deduce a time weighted estimate, which can control the growth trend of $\tau$ in $L^p$ with $p>1$. Finally, we obtain optimal time decay rate for the velocity in $L^2$ by using the time weighted estimate and a different absorption method. The result of optimal decay rate for
strong solution of the current manuscript is an extension of \cite{Luo-Yin} to the compressible fluid.
Compared to the previous literature \cite{Masmoudi2008}, the new Hardy type inequality improves the estimate for $\tau$ with $0<k\leq 1$, which will be useful for the FENE system in the future.

Our main results can be stated as follows:
\begin{theo}[Global well-posedness]\label{th1}
Let $d\geq 2~and~s>1+\frac d 2$. Let $(\rho,u,g)$ be a strong solution of \eqref{eq1} with the initial data $(\rho_0,u_0,g_0)$ satisfying the conditions $\int_B g_{0}\psi_{\infty}dR=0$ and $1+g_0>0$. Then, there exists some sufficiently small constant $\epsilon_0$ such that if
\begin{align}
E(0)=\|\rho_0\|^2_{H^s}+\|u_0\|^2_{H^s}+\|g_0\|^2_{H^s(\mathcal{L}^2)}\leq \epsilon_0,
\end{align}
then the compressible system \eqref{eq1} admits a unique global strong solution $(\rho,u,g)$ with $\int_B g\psi_{\infty}dR=0$ and $1+g>0$, and we have
\begin{align}
\sup_{t\in[0,+\infty)} E(t)+\int_{0}^{\infty}D(t)dt\leq \epsilon,
\end{align}
where $\epsilon$ is a small constant dependent on the viscosity coefficients.
\end{theo}

\begin{theo}[Large time behaviour]\label{th2}
Let $d\geq 3$. Let $(\rho,u,g)$ be a strong solution of \eqref{eq1} with the initial data $(\rho_0,u_0,g_0)$ under the condition in Theorem \ref{th1}. In addition, if $(\rho_0,u_0)\in L^1\times L^1$ and $g_0\in L^1(\mathcal{L}^p)$ with $(p-1)k>1$ and $2\leq p$, then there exists a constant $C$ such that
\begin{align}\label{decay}
\|\rho\|_{L^2}+\|u\|_{L^2}\leq C(1+t)^{-\frac d 4},
\end{align}
and
\begin{align}
\|g\|_{L^2(\mathcal{L}^2)}\leq Ce^{-Ct}.
\end{align}
\end{theo}

\begin{rema}
Taking $g \equiv 0$ and combining with the result in \cite{Li2011Large}, we can see that the $L^2$ decay rate obtained in Theorem \ref{th2} is optimal. Optimal decay rate for $(u,g)$ is an extension of \cite{Luo-Yin} to the compressible fluid.
\end{rema}

The paper is organized as follows. In Section 2 we introduce some notations and give some preliminaries which will be used in the sequel. In Section 3 we prove that the FENE dumbbell model admits a unique global strong solution provided the initial data are close to equilibrium state for $d\geq2$.
In Section 4 we study the $L^2$ decay of solutions to the compressible co-rotation FENE model for $d\geq3$ by using the linear spectral theory.
%\vspace*{2em}
%\noindent\textbf{Acknowledgements}. This work was partially supported by ...

\section{Preliminaries}
  In this section we will introduce some notations and useful lemmas which will be used in the sequel.

 If the function spaces are over $\mathbb{R}^d$ and $B$ with respect to the variable $x$ and $R$, for simplicity, we drop $\mathbb{R}^d$ and $B$ in the notation of function spaces if there is no ambiguity.

For $p\geq1$, we denote by $\mathcal{L}^{p}$ the space
$$\mathcal{L}^{p}=\big\{f \big|\|f\|^{p}_{\mathcal{L}^{p}}=\int_{B} \psi_{\infty}|f|^{p}dR<\infty\big\}.$$
We will use the notation $L^{p}_{x}(\mathcal{L}^{q})$ to denote $L^{p}[\mathbb{R}^{d};\mathcal{L}^{q}]:$
$$L^{p}_{x}(\mathcal{L}^{q})=\big\{f \big|\|f\|_{L^{p}_{x}(\mathcal{L}^{q})}=(\int_{\mathbb{R}^{d}}(\int_{B} \psi_{\infty}|f|^{q}dR)^{\frac{p}{q}}dx)^{\frac{1}{p}}<\infty\big\}.$$

The symbol $\widehat{f}=\mathcal{F}(f)$ represents the Fourier transform of $f$.
Let $\Lambda^s f=\mathcal{F}^{-1}(|\xi|^s \widehat{f})$.
If $s\geq0$, we can denote by $H^{s}(\mathcal{L}^{2})$ the space
$$H^{s}(\mathcal{L}^{2})=\{f\big| \|f\|^2_{H^{s}(\mathcal{L}^{2})}=\int_{\mathbb{R}^{d}}\int_B(|f|^2+|\Lambda^s f|^2)\psi_\infty dRdx<\infty\}.$$
Then we introduce the energy and energy dissipation functionals for $(\rho,u,g)$ as follows:
$$E(t)=\|\rho\|^2_{H^{s}}+\|u\|^2_{H^{s}}+\|g\|^2_{H^{s}(\mathcal{L}^{2})},$$
and
$$D(t)=\|\nabla\rho\|^2_{H^{s-1}}+\mu\|\nabla u\|^2_{H^{s}}+(\mu+\mu')\|{\rm div}~u\|^2_{H^{s}}+\|\nabla_R g\|^2_{H^{s}(\mathcal{L}^{2})}.$$
Sometimes we write $f\lm g$ instead of $f\leq Cg$, where $C$ is a constant. We agree that $\nabla$ stands for $\nabla_x$ and ${\rm div}$ stands for ${\rm div}_x$.

The following lemma is on various Gagliardo-Nirenberg inequalities.
\begin{lemm}\cite{1959On}\label{Lemma0}
If $d=2,~p\in[2,+\infty)$, then there exists a constant $C$ such that
 $$\|f\|_{L^{p}}\leq C \|f\|^{\frac 2 p}_{L^{2}}\|\nabla f\|^{\frac {p-2} p}_{L^{2}}.$$
For $d\geq 3$, then there exists a constant $C$ such that
 $$\|f\|_{L^{p}}\leq C \|\nabla f\|_{L^{2}}$$
where $p=\frac {2d} {d-2}$.
\end{lemm}

The following lemmas allow us to estimate the extra stress tensor $\tau$.
\begin{lemm}\cite{Masmoudi2008}\label{Lemma1}
 If $\int_B g\psi_\infty dR=0$, then there exists a constant $C$ such that
 $$\|g\|_{\mathcal{L}^{2}}\leq C \|\nabla _{R} g\|_{\mathcal{L}^{2}}.$$
\end{lemm}

\begin{lemm}\label{Lemma2}
\cite{Masmoudi2008} For all $\delta>0$, there exists a constant $C_{\delta}$ such that
$$|\tau(g)|^2\leq\delta\|\nabla _{R}g\|^2_{\mathcal{L}^{2}}
+C_{\delta}\|g\|^2_{\mathcal{L}^{2}}.$$  \\
If $(p-1)k>1$, then
$$|\tau(g)|\leq C\|g\|_{\mathcal{L}^{p}}.$$
\end{lemm}

To get $L^2$ decay rate, we prove a more precise estimate of the extra stress tensor $\tau$ by improving the method in \cite{Masmoudi2013}. Compared to the previous literature \cite{Masmoudi2008}, the results improve the estimate for $\tau$ with $0<k\leq 1$.
\begin{lemm}\label{Lemma3}
If $\|g\|_{\mathcal{L}^{2}}\leq C \|\nabla _{R} g\|_{\mathcal{L}^{2}}<\infty$, there exists a constant $C_{1}$ such that
\begin{align}\label{ha1}
|\tau(g)|\leq C_{1}\|g\|^{\frac {k+1} 2}_{\mathcal{L}^{2}}
\|\nabla _{R}g\|^{\frac {1-k} 2}_{\mathcal{L}^{2}},~~for~0<k<1,
\end{align}
and
\begin{align}\label{ha2}
|\tau(g)|\leq C_{1}\|g\|^{\frac {2n} {2n+1}}_{\mathcal{L}^{2}}
\|\nabla _{R}g\|^{\frac {1} {2n+1}}_{\mathcal{L}^{2}},~~for~k=1~and~\forall n\geq1.
\end{align}
\end{lemm}
\begin{rema}\label{rema}
According to Lemmas \ref{Lemma1}-\ref{Lemma3}, if $\int_B g\psi_\infty dR=0$, for any $k>0$, then we have $|\tau(g)|\leq C_{1}\|g\|^{\frac 1 2}_{\mathcal{L}^{2}}
\|\nabla _{R}g\|^{\frac 1 2}_{\mathcal{L}^{2}}$.
\end{rema}
\begin{proof}
We first prove the following Hardy type inequality:
\begin{align}\label{ha3}
|\tau(\psi)|=|\int_{B}(R\otimes\nabla_{R}\mathcal{U})\psi dR|\leq C_{1}(\int_B \frac {|\psi|^2} {\psi_\infty} dR)^{\frac {k+1} 4}
(\int_B \psi_\infty|\nabla_R \frac {\psi} {\psi_\infty}|^2 dR)^{\frac {1-k} 4},~~for~0<k<1.
\end{align}
Denote $x=1-|R|$. Then the proof is a simple consequence of the following 1-D inequality
\begin{align}\label{ha4}
|\int_{0}^{1} \frac {\psi} {x} dx|\leq C_{1}(\int_{0}^{1} \frac {\psi^2} {x^k} dx)^{\frac {k+1} 4}
(\int_{0}^{1} x^k|\big(\frac {\psi} {x^k}\big)'|^2 dx)^{\frac {1-k} 4},~~for~0<k<1.
\end{align}
Let $\alpha=\frac k {1-k}$. We make the following change of variables $y = x^{1-k}$, then we have $dy =(1-k)x^{-k}dx$.
We also denote $f(y) = \frac {\psi(x)} {x^k} $, $F=\int_{0}^{1}y^{2\alpha}f(y)^2 dy$ and $G=\int_{0}^{1}f'(y)^2 dy$. Hence
\begin{align*}
\int_{0}^{1} x^k|\big(\frac {\psi} {x^k}\big)'|^2 dx=(1-k)\int_{0}^{1}f'(y)^2 dy=(1-k)G.
\end{align*}
Moreover, we get
\begin{align*}
\int_{0}^{1} \frac {\psi} {x} dx=\frac 1 {1-k}\int_{0}^{1}y^{\alpha-1}f(y) dy~~and~~\int_{0}^{1} \frac {\psi^2} {x^k} dx=\frac 1 {1-k} F.
\end{align*}
For $A\in (0,1]$ which will be chosen later on, we have
\begin{align*}
\int_{0}^{1}y^{\alpha-1}f(y) dy=\int_{0}^{A}y^{\alpha-1}f(y) dy+\int_{A}^{1}y^{\alpha-1}f(y) dy.
\end{align*}
First of all, we deduce that
\begin{align*}
\int_{A}^{1}y^{\alpha-1}f(y) dy\leq (\int_{A}^{1}y^{2\alpha}f(y)^2 dy)^{\frac 1 2}(\int_{A}^{1}y^{-2}dy)^{\frac 1 2}\leq F^{\frac 1 2} A^{-\frac 1 2}.
\end{align*}
Since $G<\infty$, we obtain that $\lim_{y \rightarrow 0^{+}}y^{\alpha}f(y)=0$. Integrating by parts, we get
\begin{align*}
\int_{0}^{A}y^{\alpha-1}f(y) dy=\int_{0}^{A}\frac 1 {\alpha} (y^{\alpha})'f(y) dy=\frac 1 {\alpha} A^{\alpha}f(A)-\frac 1 {\alpha}\int_{0}^{A} y^{\alpha}f'(y) dy.
\end{align*}
Using the Cauchy-Schwarz inequality, we have
\begin{align*}
\int_{0}^{A} y^{\alpha}f'(y) dy\leq G^{\frac 1 2}(\int_{0}^{A} y^{2\alpha}dy)^{\frac 1 2}\leq CG^{\frac 1 2}A^{\alpha+\frac 1 2}.
\end{align*}
Moreover, we have
\begin{align*}
A^{\alpha}f(A)
&=A^{-1}\int_{0}^{A} (y^{\alpha+1}f(y))'dy  \\
&=A^{-1}\int_{0}^{A} (\alpha+1)y^{\alpha}f(y)dy+A^{-1}\int_{0}^{A} y^{\alpha+1}f'(y)dy  \\
&\leq CF^{\frac 1 2}A^{-\frac 1 2}+CG^{\frac 1 2}A^{\alpha+\frac 1 2}.
\end{align*}
Finally, we get
\begin{align*}
\int_{0}^{1}y^{\alpha-1}f(y) dy\leq CF^{\frac 1 2}A^{-\frac 1 2}+CG^{\frac 1 2}A^{\alpha+\frac 1 2}.
\end{align*}
Since $F\leq CG$, we complete the proof of \eqref{ha3} by choosing $A=(\frac F {CG})^{\frac {1-k} 2} \in (0,1]$. Let $\psi=g\psi_\infty$ in \eqref{ha3}. We thus get the Hardy type inequality \eqref{ha1}.  \\
The proof of \eqref{ha2} is a simple consequence of the following 1-D inequality
\begin{align}\label{ha5}
|\int_{0}^{1} \frac {\psi} {x} dx|\leq C_{1}(\int_{0}^{1} \frac {\psi^2} {x} dx)^{\frac {n} {2n+1}}
(\int_{0}^{1} x|\big(\frac {\psi} {x}\big)'|^2 dx)^{\frac {1} {4n+2}},~~for~k=1~and~\forall n\geq1.
\end{align}
We make the following change of variables $x = e^{-y}$ hence $dx =-e^{-y}dy$.
We also denote $f(y) = \frac {\psi(x)} {x} $, $F=\int_{0}^{\infty}e^{-2y}f(y)^2 dy$ and $G=\int_{0}^{\infty}f'(y)^2 dy$. Hence
\begin{align*}
\int_{0}^{1} x|\big(\frac {\psi} {x}\big)'|^2 dx=\int_{0}^{\infty}f'(y)^2 dy=G.
\end{align*}
Moreover, we get
\begin{align*}
\int_{0}^{1} \frac {\psi} {x} dx=\int_{0}^{\infty}e^{-y}f(y) dy~~and~~\int_{0}^{1} \frac {\psi^2} {x} dx= F.
\end{align*}
For $A\in [1,\infty)$ which will be chosen later on, we have
\begin{align*}
\int_{0}^{\infty}e^{-y}f(y) dy=\int_{0}^{A}e^{-y}f(y) dy+\int_{A}^{\infty}e^{-y}f(y) dy.
\end{align*}
First of all, we deduce that
\begin{align*}
\int_{0}^{A}e^{-y}f(y) dy\leq (\int_{0}^{A}e^{-2y}f(y)^2 dy)^{\frac 1 2}(\int_{0}^{A}dy)^{\frac 1 2}\leq F^{\frac 1 2} A^{\frac 1 2}.
\end{align*}
Since $G<\infty$, we obtain that $f(y)\leq C\sqrt{y}$. Integrating by parts, we get
\begin{align*}
\int_{A}^{\infty}e^{-y}f(y) dy=\int_{A}^{\infty}-(e^{-y})'f(y) dy=e^{-A}f(A)+\int_{A}^{\infty} e^{-y}f'(y) dy.
\end{align*}
Using the Cauchy-Schwarz inequality, we have
\begin{align*}
\int_{A}^{\infty} e^{-y}f'(y) dy\leq G^{\frac 1 2}(\int_{A}^{\infty} e^{-2y}dy)^{\frac 1 2}\leq CG^{\frac 1 2}e^{-A}.
\end{align*}
Moreover, we have
\begin{align*}
e^{-A}f(A)
&=e^{A}\int_{A}^{\infty} (e^{-2y}f(y))'dy  \\
&=-2e^{A}\int_{A}^{\infty} e^{-2y}f(y)dy+e^{A}\int_{A}^{\infty} e^{-2y}f'(y)dy  \\
&\leq CF^{\frac 1 2}+CG^{\frac 1 2}e^{-A}.
\end{align*}
Finally, we get
\begin{align*}
\int_{0}^{\infty}e^{-y}f(y) dy\leq CF^{\frac 1 2}A^{\frac 1 2}+CG^{\frac 1 2}e^{-A}\leq CF^{\frac 1 2}A^{\frac 1 2}+CG^{\frac 1 2}A^{-n}.
\end{align*}
Since $F\leq CG$, we complete the proof of \eqref{ha5} by choosing $A=(\frac {CG} F)^{\frac {1} {2n+1}} \in [1,\infty)$. We thus get the Hardy type inequality \eqref{ha2}.
\end{proof}
\begin{lemm}\cite{Moser1966A}\label{Lemma4}
For functions $f, g \in H^s\cap L^{\infty}$ and $s\geq 1$, we have
$$\|[\Lambda^s, f]g\|_{L^2}\leq C(\|\Lambda^{s}f\|_{L^2}\|g\|_{L^\infty}+\|\nabla f\|_{L^\infty}\|\Lambda^{s-1}g\|_{L^2}).$$
For functions $f\in H^s\cap L^{\infty}$, $g\in H^s(\mathcal{L}^{2})\cap L^{\infty}(\mathcal{L}^{2})$ and $s\geq1$, we have
$$\|[\Lambda^s, f]g\|_{L^2(\mathcal{L}^{2})}\leq C(\|\Lambda^{s}f\|_{L^2}\|g\|_{L^\infty(\mathcal{L}^{2})}+\|\nabla f\|_{L^\infty}\|\Lambda^{s-1}g\|_{L^2(\mathcal{L}^{2})}).$$
\end{lemm}

\section{Global strong solutions}
This section is devoted to investigating global strong solutions for the compressible co-rotation FENE dumbbell model with dimension $d\geq2$. To prove Theorem \ref{th1}, we divide it into two Propositions. Firstly, we give a key global priori estimate for local solutions in the following proposition.
\begin{prop}\label{pro1}
Let $d\geq 2~and~s>1+\frac d 2$. Let $(\rho_,u,g)\in L^{\infty}(0,T;H^s\times H^s\times H^s(\mathcal{L}^2))$  be local strong solutions constructed in Proposition \ref{pro2}. If $\sup_{t\in[0,T)} E(t)\leq \epsilon$, then there exist $C_0>1$ such that
\begin{align}
\sup_{t\in[0,T]} E(t)+\int_{0}^{T}D(t)dt\leq C_{0}E(0).
\end{align}
\end{prop}
\begin{proof}
Denote $L^2(\mathcal{L}^{2})$ inner product by $\langle f,g\rangle=\int_{\mathbb{R}^{d}}\int_{B}fg\psi_\infty dRdx$. Taking the $L^2(\mathcal{L}^{2})$ inner product with $g$ to $(\ref{eq1})_3$, then we have
\begin{align}
\frac {1} {2}\frac {d} {dt} \|g\|^2_{L^2(\mathcal{L}^{2})}+\|\nabla_R g\|^2_{L^2(\mathcal{L}^{2})}  =-\langle u\cdot\nabla g, g\rangle-\langle\frac 1 {\psi_\infty} \nabla_R\cdot(\sigma(u)Rg\psi_\infty), g \rangle.
\end{align}
Integrating by part, we get
\begin{align*}
-\langle u\cdot\nabla g,g\rangle=\frac 1 2 \langle {\rm div}~u, g^2\rangle
\lesssim \|\nabla u\|_{L^\infty}\|g\|^2_{L^2(\mathcal{L}^{2})}.
\end{align*}
Integrating by part and using (\ref{co}), we have
\begin{align*}
\langle\frac 1 {\psi_\infty} \nabla_R\cdot(\sigma(u)Rg\psi_\infty), g \rangle
=\frac 1 2 \int_{\mathbb{R}^{d}}\int_{B}{\rm div}_R(\sigma(u)R\psi_\infty)g^2 dRdx
=0,
\end{align*}
which implies that
\begin{align}\label{co1}
\frac {1} {2}\frac {d} {dt} \|g\|^2_{L^2(\mathcal{L}^{2})}+\|\nabla_R g\|^2_{L^2(\mathcal{L}^{2})}\lesssim \|\nabla u\|_{L^\infty}\|g\|^2_{L^2(\mathcal{L}^{2})}.
\end{align}
Multiplying $\psi_\infty$ to $(\ref{eq1})_3$ and integrating over $B$ with $R$, we deduce that $\int_{B}g\psi_\infty dR=\int_{B}g_0\psi_\infty dR=0$. Applying Lemma \ref{Lemma1}, we have $$\|g\|^2_{L^2(\mathcal{L}^{2})}\lesssim\|\nabla_R g\|^2_{L^2(\mathcal{L}^{2})}.$$

Let $h(\rho)=\frac {P'(1+\rho)} {1+\rho}$ and $i(\rho)=\frac 1 {\rho+1}$. Taking the $L^2$ inner product with $h(\rho) \rho$ to $(\ref{eq1})_1$,
then we have
\begin{align}
&\frac 1 2 \frac {d} {dt} \int_{\mathbb{R}^{d}}h(\rho)|\rho|^2 dx+\int_{\mathbb{R}^{d}}P'(1+\rho)\rho {\rm div}~u dx \\ \notag
&=\frac 1 2 \int_{\mathbb{R}^{d}} \partial_t h(\rho) |\rho|^2 dx-\int_{\mathbb{R}^{d}}h(\rho)\rho u\cdot\nabla \rho dx.
\end{align}
Taking the $L^2$ inner product with $(1+\rho)u$ to $(\ref{eq1})_2$, then we have
\begin{align}
&\frac 1 2 \frac {d} {dt} \int_{\mathbb{R}^{d}}(1+\rho)|u|^2 dx+\int_{\mathbb{R}^{d}}P'(1+\rho)u\nabla\rho dx-\int_{\mathbb{R}^{d}}u{\rm div}\Sigma(u)dx  \\ \notag
&=\int_{\mathbb{R}^{d}}u~{\rm div}~\tau dx+\frac 1 2 \int_{\mathbb{R}^{d}} \partial_t\rho|u|^2 dx-\int_{\mathbb{R}^{d}}u\cdot\nabla u (1+\rho)u dx.
\end{align}
Using integration by part, Lemmas \ref{Lemma1} and \ref{Lemma2}, we get
\begin{align*}
\int_{\mathbb{R}^{d}}u~{\rm div}~\tau dx=-\int_{\mathbb{R}^{d}}\nabla u\tau dx\leq \delta\|\nabla u\|^2_{L^2}+C_\delta\|\nabla_R g\|^2_{L^2(\mathcal{L}^{2})},
\end{align*}
where  $\delta$ is a sufficiently small constant.
Using integration by part, Lemma \ref{Lemma0}, we have
\begin{align*}
-\int_{\mathbb{R}^{d}}P'(1+\rho)(u\nabla\rho+\rho {\rm div}~u)dx
&=\int_{\mathbb{R}^{d}}P''(1+\rho)\rho u\nabla\rho dx \\
&\lesssim \|\nabla \rho\|_{L^2}\|\nabla u\|_{L^2}\|\rho\|_{L^d}+\|\nabla \rho\|^2_{L^2}\|u\|_{L^d}.
\end{align*}
The remaining terms can be treated as follows.

If $d\geq 3$, using Lemma \ref{Lemma0}, we deduce that
\begin{align*}
\frac 1 2 \int_{\mathbb{R}^{d}} \partial_t h(\rho) |\rho|^2 dx
&\lesssim \|\nabla\rho\|_{L^2}\|u\|_{L^{\frac {2d} {d-2}}}\|\rho\|_{L^d}+\|\nabla u\|_{L^2}\|\rho\|_{L^{\frac {2d} {d-2}}}\|\rho\|_{L^d}(1+\|\rho\|_{L^\infty})  \\
&\lesssim \|\nabla u\|_{L^2}\|\nabla\rho\|_{L^2}\|\rho\|_{L^d},
\end{align*}
\begin{align*}
\int_{\mathbb{R}^{d}}h(\rho)\rho u\cdot\nabla \rho dx
\lesssim \|\nabla\rho\|_{L^2}\|u\|_{L^{\frac {2d} {d-2}}}\|\rho\|_{L^d}  \lesssim \|\nabla u\|_{L^2}\|\nabla\rho\|_{L^2}\|\rho\|_{L^d},
\end{align*}
\begin{align*}
\frac 1 2 \int_{\mathbb{R}^{d}} \partial_t\rho|u|^2 dx
&\lesssim \|\nabla\rho\|_{L^2}\|u\|_{L^{\frac {2d} {d-2}}}\|u\|_{L^d}+\|\nabla u\|_{L^2}\|u\|_{L^{\frac {2d} {d-2}}}\|u\|_{L^d}(1+\|\rho\|_{L^\infty})  \\
&\lesssim \|\nabla \rho\|_{L^2}\|\nabla u\|_{L^2}\|u\|_{L^d}+\|\nabla u\|^2_{L^2}\|u\|_{L^d},
\end{align*}
\begin{align*}
\int_{\mathbb{R}^{d}}u\cdot\nabla u (1+\rho)u dx
\lesssim \|\nabla u\|_{L^2}\|u\|_{L^{\frac {2d} {d-2}}}\|u\|_{L^d}  \lesssim \|\nabla u\|^2_{L^2}\|u\|_{L^d}.
\end{align*}

If $d=2$, we verify that
\begin{align*}
\frac 1 2 \int_{\mathbb{R}^{d}} \partial_t h(\rho) |\rho|^2 dx
&\lesssim \|\nabla\rho\|_{L^2}\|u\|_{L^{4}}\|\rho\|_{L^4}+\|\nabla u\|_{L^2}\|\rho\|^2_{L^4}  \\
&\lesssim \|\nabla\rho\|_{L^2}(\|\nabla \rho\|_{L^2}\|u\|_{L^2}+\|\nabla u\|_{L^2}\|\rho\|_{L^2}),
\end{align*}
\begin{align*}
\int_{\mathbb{R}^{d}}h(\rho)\rho u\cdot\nabla \rho dx
\lesssim \|\nabla\rho\|_{L^2}\|u\|_{L^{4}}\|\rho\|_{L^4}  \lesssim \|\nabla\rho\|_{L^2}(\|\nabla u\|_{L^2}\|\rho\|_{L^2}+\|\nabla \rho\|_{L^2}\|u\|_{L^2}),
\end{align*}
\begin{align*}
\frac 1 2 \int_{\mathbb{R}^{d}} \partial_t\rho|u|^2 dx+\int_{\mathbb{R}^{d}}u\cdot\nabla u (1+\rho)u dx
&\lesssim \|\nabla\rho\|_{L^2}\|u\|^2_{L^{4}}+\|\nabla u\|_{L^2}\|u\|^2_{L^4}  \\
&\lesssim \|\nabla\rho\|_{L^2}\|\nabla u\|_{L^2}\|u\|_{L^2}+\|\nabla u\|^2_{L^2}\|u\|_{L^2}.
\end{align*}

Taking the $L^2$ inner product with $\nabla\rho$ to $(\ref{eq1})_2$, then we have
\begin{align}
&\frac {d} {dt} \int_{\mathbb{R}^{d}} u\nabla\rho dx+\gamma\|\nabla\rho\|^2_{L^2}= \int_{\mathbb{R}^{d}}u\nabla\rho_t dx  \\ \notag
&+\int_{\mathbb{R}^{d}}\nabla\rho\cdot[i(\rho) {\rm div}\Sigma{(u)}-(h(\rho)-\gamma) \nabla\rho-u\cdot\nabla u+i(\rho) {\rm div}~\tau] dx  \\ \notag
&=I_1+I_2.
\end{align}
By integration by part, we have
\begin{align*}
I_1=-\int_{\mathbb{R}^{d}}{\rm div}~u\rho_t dx
\lesssim \|\nabla u\|^2_{L^2}(1+\|\rho\|_{L^\infty})+\|\nabla u\|_{L^2}\|\nabla \rho\|_{L^2}\|u\|_{L^\infty}.
\end{align*}
Applying Lemmas \ref{Lemma1} and \ref{Lemma2}, we deduce that
\begin{align*}
I_2\lesssim \|\nabla \rho\|_{L^2}(\|\nabla^2 u\|_{L^2}+\|\rho\|_{L^\infty}\|\nabla \rho\|_{L^2}+\|u\|_{L^\infty}\|\nabla u\|_{L^2}+\|\nabla\nabla_R g\|_{L^2(\mathcal{L}^{2})}).
\end{align*}

Let $\lambda$ be a sufficiently large constant and $\eta<1$, combining all the lower order estimates for $(\ref{eq1})$, we deduce that
\begin{align}\label{low estimate}
&\frac {d} {dt} (\|h(\rho)^{\frac 1 2}\rho \|^2_{L^2}+\|(1+\rho)^{\frac 1 2}u\|^2_{L^2}+\lambda\|g\|^2_{L^2(\mathcal{L}^{2})}+2\eta\int_{\mathbb{R}^{d}} u\nabla\rho dx)  \\ \notag
&+2(\mu\|\nabla u\|^2_{L^2}+(\mu+\mu')\|{\rm div}~u\|^2_{L^2}+
\eta\gamma\|\nabla\rho\|^2_{L^2}+\lambda\|\nabla_R g\|^2_{L^2(\mathcal{L}^{2})})  \\ \notag
&\lesssim \|\nabla \rho\|_{L^2}\|\nabla u\|_{L^2}\|\rho\|_{L^d}+(\|\nabla \rho\|^2_{L^2}+\|\nabla u\|^2_{L^2})\|u\|_{L^d}+\lambda\|\nabla u\|_{L^\infty}\|\nabla_R g\|^2_{L^2(\mathcal{L}^{2})}  \\ \notag
&+\eta\|\nabla u\|^2_{L^2}(1+\|\rho\|_{L^\infty})+\delta\|\nabla u\|^2_{L^2}+C_\delta\|\nabla_R g\|^2_{L^2(\mathcal{L}^{2})}  \\ \notag
&+\eta\|\nabla \rho\|_{L^2}(\|\nabla^2 u\|_{L^2}+\|\rho\|_{L^\infty}\|\nabla \rho\|_{L^2}+\|u\|_{L^\infty}\|\nabla u\|_{L^2}+\|\nabla_R g\|_{L^2(\mathcal{L}^{2})}).
\end{align}

Now we turn to deal with the high order estimates. Applying $\Lambda^s$ to $(\ref{eq1})_3$, we infer that
\begin{align}\label{h3}
\partial_t\Lambda^s g+\mathcal{L}\Lambda^s g  =-u\cdot\nabla\Lambda^s g-[\Lambda^s,u]\nabla g-\frac 1 {\psi_\infty}\nabla_R \cdot(\Lambda^s\sigma(u)Rg\psi_\infty+R\psi_\infty[\Lambda^s,g]\sigma(u)).
\end{align}
Taking the $L^2(\mathcal{L}^{2})$ inner product with $\Lambda^s g $ to $(\ref{h3})$, then we have
\begin{align}
\frac {1} {2}\frac {d} {dt} \|\Lambda^s g\|^2_{L^2(\mathcal{L}^{2})}+\|\nabla_R  \Lambda^s g\|^2_{L^2(\mathcal{L}^{2})}
&=-\langle u\cdot\nabla \Lambda^s g,\Lambda^s g\rangle  -\langle[\Lambda^s,u]\nabla g,\Lambda^s g\rangle
\\ \notag
&-\langle\frac 1 {\psi_\infty} \nabla_R\cdot(\Lambda^s\sigma(u)Rg\psi_\infty),\Lambda^sg\rangle
\\ \notag
&-\langle\frac 1 {\psi_\infty} \nabla_R\cdot(R\psi_\infty[\Lambda^s,g]\sigma(u)),\Lambda^s g\rangle.
\end{align}
Integrating by parts and using Lemma \ref{Lemma4}, we have
\begin{align*}
-\langle u\cdot\nabla \Lambda^s g,\Lambda^s g\rangle=\frac 1 2 \langle {\rm div}~u, (\Lambda^s g)^2\rangle
\lesssim \|\nabla u\|_{L^\infty}\|g\|^2_{H^s(\mathcal{L}^{2})},
\end{align*}
and
\begin{align*}
-\langle[\Lambda^s,u]\nabla g,\Lambda^s g\rangle\lesssim \|u\|_{H^s}\|g\|^2_{H^{s}(\mathcal{L}^{2})}.
\end{align*}
Similarly, we have
\begin{align*}
\langle\frac 1 {\psi_\infty} \nabla_R\cdot(\Lambda^s\sigma(u)Rg\psi_\infty), \Lambda^sg \rangle
&=-\int_{\mathbb{R}^{d}}\int_{B}(\Lambda^s\sigma(u)R\psi_\infty g)\nabla_R \Lambda^s g dRdx \\
&\lesssim \|g\|_{L^\infty(\mathcal{L}^{2})}\|\nabla u\|_{H^s}\|\nabla_R g\|_{H^s(\mathcal{L}^{2})},
\end{align*}
and
\begin{align*}
-\langle\frac 1 {\psi_\infty} \nabla_R\cdot(R\psi_\infty[\Lambda^s,g]\sigma(u)),\Lambda^s g\rangle
&=\langle R[\Lambda^s,g]\sigma(u),\nabla_R\Lambda^s g\rangle  \\
&\lesssim \|\nabla_R  g\|_{H^s(\mathcal{L}^{2})}\|u\|_{H^s}\| g\|_{H^{s}(\mathcal{L}^{2})}.
\end{align*}
Applying Lemma \ref{Lemma1}, we deduce that
\begin{align}\label{co2}
\frac {1} {2}\frac {d} {dt} \|\Lambda^s g\|^2_{L^2(\mathcal{L}^{2})}+\|\nabla_R \Lambda^s g\|^2_{L^2(\mathcal{L}^{2})}\lesssim \|g\|_{L^\infty(\mathcal{L}^{2})}\|\nabla u\|_{H^s}\|\nabla_R g\|_{H^s(\mathcal{L}^{2})}+\|u\|_{H^s}\|\nabla_R g\|^2_{H^{s}(\mathcal{L}^{2})}.
\end{align}
Applying $\Lambda^s$ to $(\ref{eq1})_1$, we infer that
\begin{align}\label{h1}
\partial_t\Lambda^s \rho+{\rm div}~\Lambda^s u(1+\rho)
=-u\cdot\nabla\Lambda^s \rho-[\Lambda^s,u]\nabla\rho-[\Lambda^s,\rho]{\rm div}~u.
\end{align}
Taking the $L^2$ inner product with $h(\rho) \Lambda^s \rho$ to $(\ref{h1})$, then we have
\begin{align}
&\frac 1 2 \frac {d} {dt} \int_{\mathbb{R}^{d}}h(\rho)|\Lambda^s \rho|^2 dx+\int_{\mathbb{R}^{d}}P'(1+\rho)\Lambda^s \rho {\rm div}~\Lambda^s u dx =\frac 1 2
\int_{\mathbb{R}^{d}} \partial_t h(\rho) |\Lambda^s \rho|^2 dx  \\ \notag
&-\int_{\mathbb{R}^{d}}\Lambda^s \rho \cdot h(\rho)  u\cdot\nabla \Lambda^s \rho dx-\int_{\mathbb{R}^{d}}[\Lambda^s,u]\nabla\rho\cdot h(\rho) \Lambda^s\rho dx-\int_{\mathbb{R}^{d}}[\Lambda^s,(1+\rho)]{\rm div}~u\cdot h(\rho) \Lambda^s \rho dx.
\end{align}
Firstly, we obtain
\begin{align*}
\frac 1 2
\int_{\mathbb{R}^{d}} \partial_t h(\rho) |\Lambda^s \rho|^2 dx
\lesssim (\|u\|_{H^s}+\|\rho\|_{H^s})\|\Lambda^s \rho\|^2_{L^2}.
\end{align*}
By integration by part, we have
\begin{align*}
-\int_{\mathbb{R}^{d}}\Lambda^s \rho \cdot h(\rho)  u\cdot\nabla \Lambda^s \rho dx
=\frac 1 2 \int_{\mathbb{R}^{d}} {\rm div}(h(\rho)u)|\Lambda^s \rho|^2 dx
\lesssim (\|u\|_{H^s}+\|\rho\|_{H^s})\|\Lambda^s \rho\|^2_{L^2}.
\end{align*}
By the Moser-type inequality in Lemma \ref{Lemma4}, we obtain
\begin{align*}
&-\int_{\mathbb{R}^{d}}[\Lambda^s,u]\nabla\rho\cdot h(\rho) \Lambda^s\rho dx-\int_{\mathbb{R}^{d}}[\Lambda^s,\rho]{\rm div}~u\cdot h(\rho) \Lambda^s \rho dx  \\
&\lesssim \|\nabla u\|_{H^{s-1}}\|\rho\|_{H^s}\|\Lambda^s \rho\|_{L^2}.
\end{align*}
Applying $\Lambda^m$ to $(\ref{eq1})_2$, we infer that
\begin{align}\label{h2}
&\partial_t\Lambda^m u+h(\rho)\nabla\Lambda^m\rho-i(\rho) {\rm div}~\Lambda^m \Sigma{(u)}-i(\rho) {\rm div}~\Lambda^m \tau  \\ \notag
&=-u\cdot\nabla\Lambda^m u-[\Lambda^m,u]\nabla u-[\Lambda^m,h(\rho)   -\gamma]\nabla\rho\\ \notag
&+[\Lambda^m,i(\rho)-1]{\rm div}\Sigma{(u)}+[\Lambda^m,i(\rho)-1]{\rm div}~\tau.
\end{align}
Taking the $L^2$ inner product with $(1+\rho) \Lambda^s u$ to $(\ref{h2})$ with $m=s$, then we have
\begin{align}
&\frac 1 2 \frac {d} {dt} \|(1+\rho)^{\frac 1 2}\Lambda^s u\|^2_{L^2}+\int_{\mathbb{R}^{d}}P'(1+\rho)\nabla\Lambda^s \rho \Lambda^s u dx
+\mu\|\nabla\Lambda^s u\|^2_{L^2}+(\mu+\mu')\|{\rm div}~\Lambda^s u\|^2_{L^2}  \\ \notag
&=\int_{\mathbb{R}^{d}}{\rm div}~\Lambda^s \tau \Lambda^s u dx+\frac 1 2 \int_{\mathbb{R}^{d}} \partial_t\rho |\Lambda^s u|^2 dx
-\int_{\mathbb{R}^{d}}\Lambda^s u \cdot(1+\rho)u\cdot\nabla \Lambda^s u dx
\\ \notag
&-\int_{\mathbb{R}^{d}}[\Lambda^s,u]\nabla u (1+\rho)\Lambda^s u dx  -\int_{\mathbb{R}^{d}}[\Lambda^s,h(\rho)-\gamma]\nabla\rho (1+\rho)\Lambda^s u dx
\\ \notag
&+\int_{\mathbb{R}^{d}}[\Lambda^s,i(\rho)-1]{\rm div}\Sigma{(u)} (1+\rho)\Lambda^s u dx  +\int_{\mathbb{R}^{d}}[\Lambda^s,i(\rho)-1]{\rm div}~\tau (1+\rho)\Lambda^s u dx.
\end{align}
Integrating by part and using Lemmas \ref{Lemma1} and \ref{Lemma2}, we get
\begin{align*}
\int_{\mathbb{R}^{d}}{\rm div}~\Lambda^s \tau \Lambda^s u dx=-\int_{\mathbb{R}^{d}}\nabla\Lambda^s u\Lambda^s\tau dx\lesssim \delta\|\nabla\Lambda^s u\|^2_{L^2}+C_\delta\|\nabla_R \Lambda^s g\|^2_{L^2(\mathcal{L}^{2})},
\end{align*}
where  $\delta$ is a sufficiently small constant.
Using Lemma \ref{Lemma4}, we obtain
\begin{align*}
&\frac 1 2 \int_{\mathbb{R}^{d}} \partial_t\rho |\Lambda^s u|^2 dx
-\int_{\mathbb{R}^{d}}\Lambda^s u \cdot(1+\rho)u\cdot\nabla \Lambda^s u dx
-\int_{\mathbb{R}^{d}}[\Lambda^s,u]\nabla u (1+\rho)\Lambda^s u dx  \\ \notag
&\lesssim \|u\|_{H^s}\|\Lambda^s u\|_{L^2}\|\nabla u\|_{H^s},
\end{align*}
and
\begin{align*}
&-\int_{\mathbb{R}^{d}}[\Lambda^s,h(\rho)-\gamma]\nabla\rho (1+\rho)\Lambda^s u dx
+\int_{\mathbb{R}^{d}}[\Lambda^s,i(\rho)-1]{\rm div}\Sigma{(u)} (1+\rho)\Lambda^s u dx  \\ \notag
&\lesssim\|\rho\|_{H^s}\|\Lambda^s u\|_{L^2}(\|\nabla \rho\|_{H^{s-1}}+\|\nabla u\|_{H^{s}}).
\end{align*}
Using Lemmas \ref{Lemma1}, \ref{Lemma2} and \ref{Lemma4}, we obtain
\begin{align*}
\int_{\mathbb{R}^{d}}[\Lambda^s,i(\rho)-1]{\rm div}~\tau (1+\rho)\Lambda^s u dx\lesssim\|\rho\|_{H^s}\|\Lambda^s u\|_{L^2}\|\nabla_R g\|_{H^{s}(\mathcal{L}^{2})}.
\end{align*}
Integrating by part, we get
\begin{align*}
-\int_{\mathbb{R}^{d}}P'(1+\rho)(\Lambda^s u\nabla\Lambda^s\rho+\Lambda^s\rho {\rm div}~\Lambda^su)dx
&=\int_{\mathbb{R}^{d}}P''(1+\rho)\Lambda^s\rho \Lambda^s u\nabla\rho dx\\ \notag
&\lesssim \|\rho\|_{H^s}\|\Lambda^s u\|_{L^2}\|\Lambda^s \rho\|_{L^2}.
\end{align*}
Taking the $L^2$ inner product with $\nabla\Lambda^{s-1} \rho$ to $(\ref{h2})$ with $m=s-1$, then we have
\begin{align}
&\frac {d} {dt} \int_{\mathbb{R}^{d}}\Lambda^{s-1} u \cdot\nabla\Lambda^{s-1} \rho dx
+\gamma\|\nabla\Lambda^{s-1}\rho\|^2_{L^{2}}
=-\int_{\mathbb{R}^{d}} \Lambda^{s-1} \rho_t {\rm div}~\Lambda^{s-1} u dx \\ \notag
&-\int_{\mathbb{R}^{d}}\nabla\Lambda^{s-1} \rho\cdot \Lambda^{s-1}(u\cdot\nabla u) dx
-\int_{\mathbb{R}^{d}}\Lambda^{s-1}((h(\rho)-\gamma)\nabla\rho) \nabla\Lambda^{s-1} \rho dx \\ \notag
&+\int_{\mathbb{R}^{d}}\Lambda^{s-1}(i(\rho){\rm div}\Sigma{(u)}) \nabla\Lambda^{s-1} \rho dx
+\int_{\mathbb{R}^{d}}\Lambda^{s-1}(i(\rho){\rm div}~\tau) \nabla\Lambda^{s-1} \rho dx.
\end{align}
We can deduce that
\begin{align*}
&-\int_{\mathbb{R}^{d}} \Lambda^{s-1} \rho_t {\rm div}~\Lambda^{s-1} u dx
-\int_{\mathbb{R}^{d}}\nabla\Lambda^{s-1} \rho\cdot \Lambda^{s-1}(u\cdot\nabla u) dx \\
&\lesssim \|u\|_{H^s}\|\nabla u\|_{H^{s-1}}\|\nabla \rho\|_{H^{s-1}}+\|\nabla u\|^2_{H^{s-1}},
\end{align*}
and
\begin{align*}
&-\int_{\mathbb{R}^{d}}\Lambda^{s-1}((h(\rho)-\gamma)\nabla\rho) \nabla\Lambda^{s-1} \rho dx +\int_{\mathbb{R}^{d}}\Lambda^{s-1}(i(\rho){\rm div}~\Sigma{(u)}) \nabla\Lambda^{s-1} \rho dx\\
&\lesssim \|\nabla \rho\|_{H^{s-1}}(\|\nabla\rho\|_{H^{s-1}}\|\rho\|_{H^{s-1}}+\|\nabla u\|_{H^{s}}+\|\nabla u\|_{H^{s}}\|\rho\|_{H^{s-1}}).
\end{align*}
Using Lemmas \ref{Lemma1} and \ref{Lemma2}, we have
\begin{align*}
\int_{\mathbb{R}^{d}}\Lambda^{s-1}(i(\rho){\rm div}~\tau) \nabla\Lambda^{s-1} \rho dx\lesssim \|\nabla \rho\|_{H^{s-1}}\|\nabla_R  g\|_{H^{s-1}(\mathcal{L}^{2})}(\|\rho\|_{H^{s-1}}+1).
\end{align*}
Combining all the higher order derivatives estimates for $(\ref{eq1})$, we deduce that
\begin{align}\label{high estimate}
&\frac {d} {dt} (\|h(\rho)^{\frac 1 2}\Lambda^s\rho \|^2_{L^2}+\|(1+\rho)^{\frac 1 2}\Lambda^s u\|^2_{L^2}+\lambda\|\Lambda^s g\|^2_{L^2(\mathcal{L}^{2})}+2\eta\int_{\mathbb{R}^{d}} \Lambda^{s-1} u\nabla\Lambda^{s-1} \rho dx)  \\ \notag
&+2(\mu\|\nabla \Lambda^s u\|^2_{L^2}+(\mu+\mu')\|{\rm div}~\Lambda^s u\|^2_{L^2}+
\eta\gamma\|\nabla\Lambda^{s-1}\rho\|^2_{L^2}+\lambda\|\nabla_R \Lambda^s g\|^2_{L^2(\mathcal{L}^{2})})  \\ \notag
&\lesssim \|\rho\|_{H^s}\|\nabla u\|_{H^s}\|\nabla_R g\|_{H^{s}(\mathcal{L}^{2})}+(\|\rho\|_{H^{s}}+\|u\|_{H^{s}})(\|\nabla\rho\|^2_{H^{s-1}}+\|\nabla u\|^2_{H^{s}})  \\ \notag
&+\lambda(\|g\|_{L^\infty(\mathcal{L}^{2})}\|\nabla u\|_{H^s}\|\nabla_R g\|_{H^s(\mathcal{L}^{2})}+\|u\|_{H^s}\|\nabla_R g\|^2_{H^{s}(\mathcal{L}^{2})}) \\ \notag
&+\eta(\|\nabla u\|^2_{H^s}+\|\nabla u\|_{H^s}\|\nabla \rho\|_{H^{s-1}})+\delta\|\nabla \Lambda^s u\|^2_{L^2}+C_\delta\|\nabla_R \Lambda^s g\|^2_{L^2(\mathcal{L}^{2})}  \\ \notag
&+\eta\|\nabla \rho\|_{H^{s-1}}\|\nabla_R g\|_{H^s(\mathcal{L}^{2})}(1+\| \rho\|_{H^{s}}).
\end{align}
Denote that
$$E_\eta(t)=\sum_{n=0,s}(\|h(\rho)^{\frac 1 2}\Lambda^n\rho\|^2_{L^{2}}+\|(1+\rho)^{\frac 1 2}\Lambda^nu\|^2_{L^{2}})+\lambda\|g\|^2_{H^{s}(\mathcal{L}^{2})}+2\eta\sum_{m=0,s-1}\int_{\mathbb{R}^{d}} \Lambda^m u\nabla\Lambda^m \rho dx,$$
and
$$D_\eta(t)=\eta\gamma\|\nabla\rho\|^2_{H^{s-1}}+\mu\|\nabla u\|^2_{H^{s}}+(\mu+\mu')\|{\rm div}~u\|^2_{H^{s}}+\lambda\|\nabla_R g\|^2_{H^{s}(\mathcal{L}^{2})}.$$
For some sufficiently small constant $\eta>0$, we obtain $E(t)\sim E_\eta(t)$ and $D(t)\sim D_\eta(t)$. Therefore, the small assumption $E(t)\leq\epsilon$ implies that $E_\eta(t)\lesssim \epsilon$.
Then combining the estimates \eqref{low estimate} and \eqref{high estimate}, we finally infer that
$$\frac d {dt} E_\eta(t)+2D_\eta(t)\lesssim (\epsilon^{\frac 1 2}+\delta+\eta^{\frac 1 2})D_\eta(t)+C_\delta\|\nabla_R \Lambda^s g\|^2_{L^2(\mathcal{L}^{2})}.$$
Choosing some sufficiently small fixed $\delta, \eta>0$ and sufficiently large $\lambda>C_\delta$, if $\epsilon$ is small enough, then we have
\begin{align*}
\sup_{t\in[0,T]}E_\eta(t)+\int_{0}^{T}D_\eta(t)dt\leq E_\eta(0).
\end{align*}
Using the equivalence of $E(t)\sim E_\eta(t)$ and $D(t)\sim D_\eta(t)$, then there exist $C_0>1$ such that
\begin{align*}
\sup_{t\in[0,T]} E(t)+\int_{0}^{T}D(t)dt\leq C_{0}E(0).
\end{align*}
We thus complete the proof of Proposition \ref{pro1}.
\end{proof}
Now, we need to prove the existence of local solutions in some appropriate spaces by a standard iterating method.
\begin{prop}\label{pro2}
Let $d\geq 2~and~s>1+\frac d 2$. Assume $E(0)\leq \frac {\epsilon} 2$. Then there exist a time $T>0$ such that the FENE polymeric system \eqref{eq1} admits a unique local strong solution $(\rho_,u,g)\in L^{\infty}(0,T;H^s\times H^s\times H^s(\mathcal{L}^2))$ and we have
\begin{align}
\sup_{t\in[0,T]} E(t)+\int_{0}^{T}H(t)dt\leq \epsilon,
\end{align}
where $H(t)=\mu\|\nabla u\|^2_{H^{s}}+(\mu+\mu')\|{\rm div}~u\|^2_{H^{s}}+\|\nabla_R g\|^2_{H^{s}(\mathcal{L}^{2})}.$
\end{prop}
\begin{proof}
First, we consider the iterating approximating sequence as follows:
\begin{align}\label{approximate}
\left\{
\begin{array}{ll}
\rho^{n+1}_t+(1+\rho^{n}){\rm div}~u^{n+1}=-u^{n}\cdot\nabla\rho^{n+1} , \\[1ex]
u^{n+1}_t-\frac 1 {1+\rho^{n}} {\rm div}\Sigma{(u^{n+1})}+\frac {P'(1+\rho^{n})} {1+\rho^{n}} \nabla\rho^{n+1}=-u^{n}\cdot\nabla u^{n+1}+\frac 1 {1+\rho^{n}} {\rm div}~\tau^{n+1}, \\[1ex]
g^{n+1}_t+\mathcal{L}g^{n+1}=-u^{n}\cdot\nabla g^{n+1}-\frac 1 {\psi_\infty}\nabla_R\cdot(\sigma(u^{n})Rg^{n+1}\psi_\infty),  \\[1ex]
\end{array}
\right.
\end{align}
with the initial data $(\rho^{n+1}, u^{n+1}, g^{n+1})|_{t=0}=(\rho_0(x), u_0(x), g_0(x,R))$. Furthermore, start with   \\
$(\rho^{0}(t,x), u^{0}(t,x), g^{0}(t,x,R))=(0, 0, 0)$, we obtain the approximate sequence $(\rho^n, u^n, g^n)$.

Let $E^n(t)=E(\rho^n, u^n, g^n)(t)$ and $H^n(t)=H(\rho^n, u^n, g^n)(t)$. We claim that: there exist $\epsilon>0$ and $T>0$ such that for any $n\in N$, if $E(0)\leq \frac \epsilon 2$ and $\sup_{t\in [0,T]}E^n(t)+\int_{0}^{T}H^n(t)dt\leq\epsilon$, then
\begin{align}\label{uniform}
\sup_{t\in [0,T]}E^{n+1}(t)+\int_{0}^{T}H^{n+1}(t)dt\leq\epsilon.
\end{align}
Similar to the estimates \eqref{low estimate} and \eqref{high estimate}, using Lemmas \ref{Lemma0}-\ref{Lemma2} and \ref{Lemma4}, we obtain
\begin{align*}
\frac 1 2 \frac d {dt} \|\rho^{n+1}\|^2_{H^s}&\lesssim \|u^{n}\|_{H^s}\|\rho^{n+1}\|^2_{H^s}+\|\nabla u^{n+1}\|_{H^s}\|\rho^{n+1}\|_{H^s}  \\
&+\|\rho^{n}\|_{H^s}\|\rho^{n+1}\|_{H^s}(\|\nabla u^{n+1}\|_{H^s}+\|u^{n+1}\|_{H^s}),
\end{align*}
and
\begin{align*}
&\frac 1 2 \frac d {dt} \|u^{n+1}\|^2_{H^s}+\mu\|\nabla u^{n+1}\|^2_{H^{s}}+(\mu+\mu')\|{\rm div}~u^{n+1}\|^2_{H^{s}}  \\
&\lesssim \|u^{n}\|_{H^s}\|u^{n+1}\|^2_{H^s}+\|\nabla u^{n+1}\|_{H^s}\|\rho^{n}\|_{H^s}+\|\rho^{n}\|_{H^s}\|u^{n+1}\|_{H^s}\|\rho^{n+1}\|_{H^s}  \\
&+\|\rho^{n}\|_{H^s}\|\nabla u^{n+1}\|_{H^s}(\|u^{n+1}\|_{H^s}+\|\nabla_R g^{n+1}\|_{H^{s}(\mathcal{L}^{2})})  \\
&+\delta \|\nabla u^{n+1}\|_{H^s}\|\nabla_R g^{n+1}\|_{H^{s}(\mathcal{L}^{2})}+C_\delta\|\nabla u^{n+1}\|_{H^s}\|g^{n+1}\|_{H^{s}(\mathcal{L}^{2})}.
\end{align*}
We can infer from $(\ref{approximate})_3$ that
\begin{align*}
\frac 1 2 \frac d {dt} \|g^{n+1}\|^2_{H^{s}(\mathcal{L}^{2})}+\|\nabla_R g^{n+1}\|^2_{H^{s}(\mathcal{L}^{2})}
&\lesssim (\|u^{n}\|_{H^s}+\|\nabla u^{n}\|_{H^s})\|g^{n+1}\|^2_{H^{s}(\mathcal{L}^{2})}   \\
&+\|\nabla u^{n}\|_{H^s}\|g^{n+1}\|_{H^{s}(\mathcal{L}^{2})}\|\nabla_R g^{n+1}\|_{H^{s}(\mathcal{L}^{2})}.
\end{align*}
Adding up the above estimates, we obtain
\begin{align*}
\frac 1 2 \frac d {dt} E^{n+1}(t)+H^{n+1}(t)
&\leq C[(E^{n})^{\frac 1 2}E^{n+1}+((E^{n})^{\frac 1 2}+(E^{n+1})^{\frac 1 2})(H^{n+1})^{\frac 1 2}\\
&+(E^{n})^{\frac 1 2}(E^{n+1})^{\frac 1 2}(H^{n+1})^{\frac 1 2}  +(E^{n})^{\frac 1 2}H^{n+1}+\delta H^{n+1}\\
&+(H^{n})^{\frac 1 2}E^{n+1}+(E^{n+1})^{\frac 1 2}(H^{n})^{\frac 1 2}(H^{n+1})^{\frac 1 2}].
\end{align*}
Using the induction assumptions $\sup_{t\in [0,T]}E^n(t)+\int_{0}^{T}H^n(t)dt\leq\epsilon$, we get
$$((E^{n})^{\frac 1 2}+(E^{n+1})^{\frac 1 2})(H^{n+1})^{\frac 1 2}\leq \frac 1 {8C}H^{n+1}+2C(E^{n+1}+\epsilon),$$
$$(E^{n})^{\frac 1 2}(E^{n+1})^{\frac 1 2}(H^{n+1})^{\frac 1 2}\leq \epsilon^{\frac 1 2} H^{n+1}+\epsilon^{\frac 1 2} E^{n+1},$$
$$(E^{n})^{\frac 1 2}H^{n+1}\leq \epsilon^{\frac 1 2} H^{n+1},$$
$$(H^{n})^{\frac 1 2}E^{n+1}\leq(1+H^{n})E^{n+1},$$
$$(E^{n+1})^{\frac 1 2}(H^{n})^{\frac 1 2}(H^{n+1})^{\frac 1 2}\leq \frac 1 {8C}H^{n+1}+2CE^{n+1}H^{n}.$$
Let $\mathcal{E}_n(t)=E^{n}(t)+\int_{0}^{t}H^{n}(s)ds$. Then we deduce that
\begin{align*}
&\frac 1 2 E^{n+1}(t)+(\frac 3 4 -C \delta-C\epsilon^{\frac 1 2})\int_{0}^{t}H^{n+1}(s)ds  \\
&\leq \frac 1 2 E^{n+1}(0)+Ct\epsilon+Ct(1+\epsilon^{\frac 1 2})\sup_{s\in[0,t]}\mathcal{E}_{n+1}(s)+C\epsilon\sup_{s\in[0,t]}\mathcal{E}_{n+1}(s).
\end{align*}
If constant $\delta,~\epsilon^{\frac 1 2}\leq \frac 1 {32C}$, for $t\leq T$, then
\begin{align*}
&(\frac 1 2-CT(1+\epsilon^{\frac 1 2})-C\epsilon) \sup_{s\in[0,t]}\mathcal{E}_{n+1}(s) \leq \frac 1 4 \epsilon+CT\epsilon,
\end{align*}
which implies that
\begin{align*}
&\sup_{t\in[0,T]}\mathcal{E}_{n+1}(t) \leq \epsilon,
\end{align*}
where $T\leq \epsilon^{\frac 1 2}\leq \frac 1 {32C}$. The claim \eqref{uniform} is true. We thus complete the proof of uniform bound for $(\rho^n, u^n, g^n)$.

We now prove the convergence in a low norm by using \eqref{uniform} with $n\geq 0$. Let $\tilde{\rho}^{n+1}=\rho^{n+1}-\rho^{n}$, $\tilde{u}^{n+1}=u^{n+1}-u^{n}$ and $\tilde{g}^{n+1}=g^{n+1}-g^{n}$ , then it follows from \eqref{approximate} that
\begin{align}\label{convergence}
\left\{
\begin{array}{ll}
\tilde{\rho}^{n+1}_t+(1+\rho^{n}){\rm div}~\tilde{u}^{n+1}=-u^{n}\cdot\nabla\tilde{\rho}^{n+1}-\tilde{u}^{n}\cdot\nabla\rho^{n}-\tilde{\rho}^{n}{\rm div}~u^{n}, \\
\tilde{u}^{n+1}_t-i(\rho^{n}){\rm div}\Sigma{(\tilde{u}^{n+1})}-(i(\rho^{n})-i(\rho^{n-1})){\rm div}\Sigma{(u^{n})}\\[1ex]
+h(\rho^{n}) \nabla\tilde{\rho}^{n+1}+(h(\rho^{n})-h(\rho^{n-1}))\nabla\rho^{n}  \\ [1ex]
=-u^{n}\cdot\nabla \tilde{u}^{n+1}-\tilde{u}^{n}\cdot\nabla u^{n}+i(\rho^{n}) {\rm div}~\tilde{\tau}^{n+1}+(i(\rho^{n})-i(\rho^{n-1})){\rm div}~\tau^{n}, \\[1ex]
\tilde{g}^{n+1}_t+\mathcal{L}\tilde{g}^{n+1}=-u^{n}\cdot\nabla \tilde{g}^{n+1}-\tilde{u}^{n}\cdot\nabla g^{n} \\[1ex]
-\frac 1 {\psi_\infty}\nabla_R\cdot(\sigma(u^{n})R\tilde{g}^{n+1}\psi_\infty)-\frac 1 {\psi_\infty}\nabla_R\cdot(\sigma(\tilde{u}^{n})Rg^{n}\psi_\infty).  \\[1ex]
\end{array}
\right.
\end{align}
By the standard energy estimate, we get
\begin{align*}
\frac 1 2 \frac d {dt} \|\tilde{\rho}^{n+1}\|^2_{L^2}&\lesssim \|\nabla u^{n}\|_{L^\infty}\|\tilde{\rho}^{n+1}\|^2_{L^2}+(1+\|\rho^{n}\|_{L^\infty})\|\nabla \tilde{u}^{n+1}\|_{L^2}\|\tilde{\rho}^{n+1}\|_{L^2}  \\
&+\|\nabla\rho^{n}\|_{L^\infty}\|\tilde{u}^{n}\|_{L^2}\|\tilde{\rho}^{n+1}\|_{L^2}+\|\nabla u^{n}\|_{L^\infty}\|\tilde{\rho}^{n}\|_{L^2}\|\tilde{\rho}^{n+1}\|_{L^2} \\
&\lesssim \epsilon^{\frac 1 2 }\|\tilde{\rho}^{n+1}\|^2_{L^2}+\|\nabla \tilde{u}^{n+1}\|_{L^2}\|\tilde{\rho}^{n+1}\|_{L^2}+ \epsilon^{\frac 1 2 }\|\tilde{\rho}^{n+1}\|_{L^2}(\|\tilde{u}^{n}\|_{L^2}+\|\tilde{\rho}^{n}\|_{L^2}),
\end{align*}
and
\begin{align*}
&\frac 1 2 \frac d {dt} \|\tilde{u}^{n+1}\|^2_{L^2}+\mu\|\nabla \tilde{u}^{n+1}\|^2_{L^2}+(\mu+\mu')\|{\rm div}~\tilde{u}^{n+1}\|^2_{L^2}  \\
&\lesssim \|\nabla u^{n}\|_{L^\infty}\|\tilde{u}^{n+1}\|^2_{L^2}+\|\nabla \tilde{u}^{n+1}\|_{L^2}\|\tilde{\rho}^{n+1}\|_{L^2}+\|\nabla u^{n}\|_{L^\infty}\| \tilde{u}^{n}\|_{L^2}\|\tilde{u}^{n+1}\|_{L^2}  \\
&+\|\nabla \rho^{n}\|_{L^\infty}\| \tilde{u}^{n+1}\|_{L^2}(\|\tilde{\rho}^{n+1}\|_{L^2}+\|\tilde{\rho}^{n}\|_{L^2})+\|{\rm div}\Sigma{(u^{n})}\|_{L^\infty}\| \tilde{u}^{n+1}\|_{L^2}\|\tilde{\rho}^{n}\|_{L^2}  \\
&+\|\nabla\rho^{n}\|_{L^\infty}\| \tilde{u}^{n+1}\|_{L^2}(\|\nabla\tilde{u}^{n+1}\|_{L^2}+\|\nabla_R \tilde{g}^{n+1}\|_{L^{2}(\mathcal{L}^{2})})+\|\nabla\nabla_R g^{n}\|_{L^{\infty}(\mathcal{L}^{2})}\| \tilde{u}^{n+1}\|_{L^2}\|\tilde{\rho}^{n}\|_{L^2}  \\
&+\delta \|\nabla \tilde{u}^{n+1}\|_{L^2}\|\nabla_R \tilde{g}^{n+1}\|_{L^{2}(\mathcal{L}^{2})}+C_\delta\|\nabla \tilde{u}^{n+1}\|_{L^2}\|\tilde{g}^{n+1}\|_{L^{2}(\mathcal{L}^{2})}\\
&\lesssim \epsilon^{\frac 1 2 }(\|\tilde{\rho}^{n+1}\|^2_{L^2}+\|\tilde{u}^{n+1}\|^2_{L^2}+\|\tilde{\rho}^{n}\|^2_{L^2}+\|\tilde{u}^{n}\|^2_{L^2})+\|\nabla \tilde{u}^{n+1}\|_{L^2}(\|\tilde{\rho}^{n+1}\|_{L^2}+\|\tilde{g}^{n+1}\|_{L^{2}(\mathcal{L}^{2})})  \\
&+(\|{\rm div}\Sigma{(u^{n})}\|_{L^\infty}+\|\nabla\nabla_R g^{n}\|_{L^{\infty}(\mathcal{L}^{2})})\| \tilde{u}^{n+1}\|_{L^2}\|\tilde{\rho}^{n}\|_{L^2}+\delta \|\nabla \tilde{u}^{n+1}\|_{L^2}\|\nabla_R \tilde{g}^{n+1}\|_{L^{2}(\mathcal{L}^{2})}  \\
&+\epsilon^{\frac 1 2 }\| \tilde{u}^{n+1}\|_{L^2}(\|\nabla\tilde{u}^{n+1}\|_{L^2}+\|\nabla_R \tilde{g}^{n+1}\|_{L^{2}(\mathcal{L}^{2})}).
\end{align*}
We can infer from $(\ref{convergence})_3$ that
\begin{align*}
\frac 1 2 \frac d {dt} \|\tilde{g}^{n+1}\|^2_{L^{2}(\mathcal{L}^{2})}+\|\nabla_R \tilde{g}^{n+1}\|^2_{L^{2}(\mathcal{L}^{2})}
&\lesssim \|\nabla u^{n}\|_{L^\infty}\|\tilde{g}^{n+1}\|^2_{L^2(\mathcal{L}^{2})}   \\
&+\|\tilde{u}^{n}\|_{L^2}\|\nabla g^{n}\|_{L^\infty(\mathcal{L}^{2})}\|\tilde{g}^{n+1}\|_{L^2(\mathcal{L}^{2})}\\
&+\|\nabla \tilde{u}^{n}\|_{L^2}\|g^{n}\|_{L^\infty(\mathcal{L}^{2})}\|\nabla_R \tilde{g}^{n+1}\|_{L^2(\mathcal{L}^{2})}  \\
&\lesssim \epsilon^{\frac 1 2 }(\|\tilde{g}^{n+1}\|^2_{L^2(\mathcal{L}^{2})}+\|\tilde{u}^{n}\|^2_{L^2}+\|\nabla \tilde{u}^{n}\|_{L^2}\|\nabla_R \tilde{g}^{n+1}\|_{L^2(\mathcal{L}^{2})}).
\end{align*}
Then we introduce some functionals for $(\rho,u,g)$ as follows:
$$E_1(t)=\|\rho\|^2_{L^{2}}+\|u\|^2_{L^{2}}+\|g\|^2_{L^{2}(\mathcal{L}^{2})},$$
and
$$H_1(t)=\mu\|\nabla u\|^2_{L^{2}}+(\mu+\mu')\|{\rm div}~u\|^2_{L^{2}}+\|\nabla_R g\|^2_{L^{2}(\mathcal{L}^{2})}.$$
Denote $\widetilde{E}^n(t)=E_1(\tilde{\rho}^n,\tilde{u}^n,\tilde{g}^n)(t)$, $\widetilde{H}^n(t)=H_1(\tilde{\rho}^n,\tilde{u}^n,\tilde{g}^n)(t)$ and $\widetilde{\mathcal{E}}^n(t)=\widetilde{E}^n(t)+\int_{0}^{t}\widetilde{H}^n(s)ds$. Combining the above estimates, we have
\begin{align*}
\frac 1 2 \frac d {dt} \widetilde{E}^{n+1}+\widetilde{H}^{n+1}
&\leq C\epsilon^{\frac 1 2 }( \widetilde{E}^{n+1}+\widetilde{H}^{n+1}  + \widetilde{E}^{n}+\widetilde{H}^{n})\\
&+(\frac 1 4+\delta)\widetilde{H}^{n+1}+C\widetilde{E}^{n+1}+(H^n)^{\frac 1 2}(\widetilde{E}^{n+1}+\widetilde{E}^{n}),
\end{align*}
where $H^n(t)=\mu\|\nabla u^n\|^2_{H^{s}}+(\mu+\mu')\|{\rm div}~u^n\|^2_{H^{s}}+\|\nabla_R g^n\|^2_{H^{s}(\mathcal{L}^{2})}.$ Since $\widetilde{E}^{n}(0)=0$, we have
\begin{align*}
(\frac 1 2-CT-C\epsilon^{\frac 1 2 })\sup_{t\in[0,T]}\widetilde{\mathcal{E}}^{n+1}(t)\leq (C\epsilon^{\frac 1 2 }T+C\epsilon^{\frac 1 2 })\sup_{t\in[0,T]}\widetilde{\mathcal{E}}^n(t)
+\int_{0}^{T}(H^n)^{\frac 1 2}(\widetilde{E}^{n+1}+\widetilde{E}^{n})dt,
\end{align*}
which implies that
\begin{align*}
(\frac 1 2-CT-CT^{\frac 1 2 }\epsilon^{\frac 1 2 }-C\epsilon^{\frac 1 2 })\sup_{t\in[0,T]}\widetilde{\mathcal{E}}^{n+1}(t)\leq (C\epsilon^{\frac 1 2 }T+CT^{\frac 1 2 }\epsilon^{\frac 1 2 }+C\epsilon^{\frac 1 2 })\sup_{t\in[0,T]}\widetilde{\mathcal{E}}^n(t).
\end{align*}
Choosing sufficiently small $\epsilon$, $T$ and using a standard compactness argument, we finally get a unique local solution of the compressible FENE system \eqref{eq1} with $\sup_{t\in[0,T]} E(t)+\int_{0}^{T}H(t)dt\leq \epsilon$. We thus complete the proof of Proposition \ref{pro2}.
\end{proof}
{\bf The proof of Theorem \ref{th1}:}

Combining Propositions \ref{pro1} and \ref{pro2}, we prove the global-in-time solutions of \eqref{eq1} by using the bootstrap argument. We assume that the initial datum satisfies $E(0)\leq \epsilon_0$ with $\epsilon_0=\frac {\epsilon} {2C_0}$ and $C_0>1$. Applying Proposition \ref{pro2}, we have the unique local solution on $[0,T]$ with $T>0$, and $\sup_{t\in[0,T]} E(t)\leq \epsilon$. Using global priori estimate stared in Proposition \ref{pro1}, we obtain
$$E(T)\leq C_0 E(0)\leq \frac {\epsilon} {2}. $$
Applying Proposition \ref{pro2} again, we get the unique local solution result on the time interval $t\in[T,2T]$, satisfying $\sup_{t\in[T,2T]} E(t)\leq \epsilon$. So we have
$$\sup_{t\in[0,2T]} E(t)\leq \epsilon.$$
Then global priori estimate stared in Proposition \ref{pro1} yields
$$E(2T)\leq C_0 E(0)\leq \frac {\epsilon} {2}. $$
Repeating this process, we show global existence of strong solution for \eqref{eq1} near equilibrium. Furthermore, we obtain $\sup_{t\in[0,\infty)} E(t)+\int_{0}^{\infty}D(t)dt\leq C_{0}E(0).$
\hfill$\Box$

\section{Optimal decay rate}
This section is devoted to investigating optimal decay rate of global strong solutions for the compressible co-rotation FENE dumbbell model with dimension $d\geq3$.
In order to obtain the optimal $L^2$ decay rate, we will use the linear spectral theory. For this purpose, we rewrite the first two equations of \eqref{eq1} so that all the nonlinear terms $F(U)$ and the external term $i(\rho) {\rm div}~\tau$ appear at the right hand side of equations:
\begin{align}\label{eq2}
\left\{
\begin{array}{ll}
\rho_t+{\rm div}~u=-u\cdot\nabla\rho-\rho {\rm div}~u , \\[1ex]
u_t-{\rm div}\Sigma{(u)}+\gamma \nabla\rho=-u\cdot\nabla u+(i(\rho)-1){\rm div}\Sigma{(u)}+(\gamma-h(\rho))\nabla\rho+i(\rho) {\rm div}~\tau,
\end{array}
\right.
\end{align}
where $h(\rho)=\frac {P'(1+\rho)} {1+\rho}$ and $i(\rho)=\frac 1 {\rho+1}$.
Taking the Fourier transform to system \eqref{eq2}, we have
\begin{align}\label{eq3}
\hat{U}_t-A(\xi)\hat{U}=\widehat{F(U)}+\mathcal{F}_x(0,{\rm div}(i(\rho)\tau))^T-\mathcal{F}_x(0,\tau\nabla i(\rho) )^T,
\end{align}
where $A(\xi)$ is defined as
\begin{equation} %开始数学环境
\left(
\begin{array}{cc}
0 & -i\xi^T \\
-i\xi\gamma & -\mu|\xi|^2 Id-(\mu+\mu')\xi\otimes\xi \\
\end{array}
\right).
\end{equation}
By some simple calculation, we get the determinant
$$det(A-\lambda Id)=(\lambda+\mu|\xi|^2)^{d-1}(\lambda^2+(2\mu+\mu')|\xi|^2 \lambda+\gamma|\xi|^2).$$
The eigenvalues $\lambda_j$ of $A(\xi)$ and their projections $P_j$ are analyzed by
\begin{lemm}\cite{Matsumura,Li2011Large}\label{Lemma5}
(1) $\lambda_j$ depends on $i|\xi|$ only and if $|\xi|=0$, then $\lambda_j=0$. \\
(2)The eigenvalues $\lambda_j$ of $A(\xi)$ can be computed as
\begin{align}\label{eq4}
\left\{
\begin{array}{ll}
\lambda_0=-\mu|\xi|^2, \\[1ex]
\lambda_+=-(\mu+\frac 1 2 \mu')|\xi|^2+\frac 1 2 i\sqrt{4\gamma|\xi|^2-(2\mu+\mu')^2|\xi|^4}, \\[1ex]
\lambda_-=-(\mu+\frac 1 2 \mu')|\xi|^2-\frac 1 2 i\sqrt{4\gamma|\xi|^2-(2\mu+\mu')^2|\xi|^4}.
\end{array}
\right.
\end{align}
For $a.e.$ $|\xi|>0$, we have $rank(\lambda_0 Id-A(\xi))=2$.  \\
(3)The semigroup $e^{tA}$ is expressed as
$$e^{tA}=e^{t\lambda_0}P_0+e^{t\lambda_+}P_+ +e^{t\lambda_-}P_- ,$$
where the project $P_j$ can be computed as
$$P_j=\prod_{i\neq j} \frac {A(\xi)-\lambda_i I} {\lambda_j-\lambda_i}.$$
(4)If $|\xi|\leq r_1$, for a positive constant $\beta_{j}$, we have
$\|P_j\|_{L^\infty}\leq C$ and $e^{t\lambda_j}\sim e^{-\beta_{j} t|\xi|^2}$.  \\
If $|\xi|> r_1$, for a positive constant $\beta_2$, we deduce that
$\|e^{tA}\|_{L^\infty}\leq Ce^{-\beta_2 t}$.
\end{lemm}

From \eqref{eq3}, we see that there is a linear term $\mathcal{F}_x(0,{\rm div}(i(\rho)\tau) )^T\sim \mathcal{F}_x(0,{\rm div}~\tau)^T$ in the right hand side. Therefore, we need to estimate the $L^1$-norm of the stress tensor $\tau$.
\begin{prop}
Under the condition in Theorem \ref{th2}. There exists a constant $C$ such that for any $t>0$, we have
\begin{align}\label{tau1}
\|\tau\|_{L^1}\lesssim \|g\|_{L^1_x(\mathcal{L}^p)}\lesssim e^{C\sqrt{t}}\|g_0\|_{L^1_x(\mathcal{L}^p)}.
\end{align}
\begin{proof}
Multiplying $p|g|^{p-2}g\psi_\infty$ by both sides of (\ref{eq1}) and integrating over $B$ with $R$, we obtain
\begin{align*}
&\frac{d}{dt}\int_B|g|^p \psi_\infty dR+\frac{4(p-1)}{p}\int_B\psi_\infty|\nabla_R(g ^{\frac{p}{2}})|^2 dR \\ \notag
&=-u\cdot\nabla_x\int_B|g|^p \psi_\infty dR+(p-1)\int_B\sigma(u)R\psi_\infty\nabla_R|g|^p dR.
\end{align*}
Using integration by parts and (\ref{co}), we have
\begin{align*}
\frac{d}{dt}\int_B|g|^p \psi_\infty dR+\frac{4(p-1)}{p}\int_B\psi_\infty|\nabla_R(g ^{\frac{p}{2}})|^2 dR
=-u\cdot\nabla_x\int_B|g|^p \psi_\infty dR,
\end{align*}
which implies that
\begin{align*}
\frac{d}{dt}\int_B|g|^p \psi_\infty dR\leq-u\cdot\nabla_x\int_B|g|^p \psi_\infty dR.
\end{align*}
Multiplying $\|g\|^{1-p}_{\mathcal{L}^p}$ by both sides of the above inequality, integrating over $\mathbb{R}^{d}$ with $x$ and applying Gronwall's inequality, we obtain
$$\|g\|_{L^1_x(\mathcal{L}^p)}\lesssim e^{C\int_{0}^{t}\|\nabla u\|_{L^\infty}ds}\|g_0\|_{L^1_x(\mathcal{L}^p)}\lesssim e^{C\sqrt{t}}\|g_0\|_{L^1_x(\mathcal{L}^p)}.$$
Applying Lemma \ref{Lemma2} and using the fact that $(p-1)k>1$,
\begin{align*}
\|\tau\|_{L^1}\lesssim \|g\|_{L^1_x(\mathcal{L}^p)}\lesssim e^{C\sqrt{t}}\|g_0\|_{L^1_x(\mathcal{L}^p)}.
\end{align*}
\end{proof}
\end{prop}

\begin{rema}
For the incompressible FENE model, one can show that $\|\tau\|_{L^1}\leq C$. However, in the compressible case, the $L^1$-norm of $\tau$ has a exponential growth. This is the main difference between the incompressible and compressible model.
\end{rema}

{\bf The proof of Theorem \ref{th2}:}
Observe that \eqref{co1} and Theorem \ref{th1} ensures
\begin{align}
\frac {d} {dt} \|g\|^2_{L^2(\mathcal{L}^{2})}+\|\nabla_R g\|^2_{L^2(\mathcal{L}^{2})}\leq 0.
\end{align}
Using Lemma \ref{Lemma1}, we deduce that
\begin{align}
\|g\|_{L^2(\mathcal{L}^2)}\leq \|g_0\|_{L^2(\mathcal{L}^2)}e^{-Ct}.
\end{align}
According to \eqref{eq3}, we obtain
\begin{align}\label{eq5}
\hat{U}=e^{tA}\hat{U}(0,\xi)+\int_{0}^{t}e^{(t-s')A}(\widehat{F(U)}+\mathcal{F}_x(0,{\rm div}(i(\rho)\tau))^T+\mathcal{F}_x(0,\tau\nabla i(\rho) )^T )ds'.
\end{align}
Under the additional assumption $U(0,x)\in (L^1)^2$, using Lemma \ref{Lemma5}, we get
\begin{align}\label{ineq1}
\|e^{tA}\hat{U}(0,\xi)\|_{L^2}
&\leq (\int_{|\xi|\leq r_1}e^{2tA}|\hat{U}(0,\xi)|^2 d\xi)^{\frac 1 2}+(\int_{|\xi|> r_1}e^{2tA}|\hat{U}(0,\xi)|^2 d\xi)^{\frac 1 2}  \\ \notag
&\leq (\int_{|\xi|\leq r_1}e^{-2\beta_{j} t|\xi|^2} d\xi)^{\frac 1 2}\|\hat{U}(0,\xi)\|_{L^\infty}+Ce^{-\beta_2 t}\|\hat{U}(0,\xi)\|_{L^2}  \\ \notag
&\leq C(1+t)^{-\frac d 4}(\|U(0,x)\|_{L^1}+\|U(0,x)\|_{L^2}).
\end{align}
For the nonlinear terms $F(U)$ and $\delta=\frac 1 2 (s-1-\frac d 2)>0$, we can easily deduce that
\begin{align}\label{ineq2}
\|\int_{0}^{t}e^{(t-s')A}\widehat{F(U)}ds'\|_{L^2}
&\leq C\int_{0}^{t}(1+ t-s')^{-\frac d 4}(\|F(U)\|_{L^1}+\|F(U)\|_{L^2})ds'  \\ \notag
&\leq C\int_{0}^{t}(1+ t-s')^{-\frac d 4}(\|U\|_{L^2}\|\nabla U\|_{H^1}+\|U\|_{L^\infty}\|\nabla U\|_{H^{1}})ds'  \\ \notag
&\leq C\epsilon^{\frac 1 2} (\int_{0}^{t}(1+ t-s')^{-\frac d 2}\|U\|^2_{H^{s-1-\delta}}ds')^{\frac 1 2}.
\end{align}
Using Lemma \ref{Lemma5}, Remark \ref{rema} and Theorem \ref{th1}, we get
\begin{align}\label{ineq3}
\|\int_{0}^{t}e^{(t-s')A}\mathcal{F}_x(0,\tau\nabla i(\rho) )^T ds'\|_{L^2}
&\leq C\int_{0}^{t}(1+ t-s')^{-\frac d 4}(\|\tau\nabla i(\rho)\|_{L^1}+\|\tau\nabla i(\rho)\|_{L^2})ds'  \\ \notag
&\leq C\int_{0}^{t}(1+ t-s')^{-\frac d 4}\|\tau\|_{L^{2}}\|\nabla \rho\|_{H^{s-1}}ds'  \\ \notag
&\leq C(\int_{0}^{t}(1+ t-s')^{-\frac d 2}\|g\|_{L^2(\mathcal{L}^2)}\|\nabla_R g\|_{L^2(\mathcal{L}^2)}ds')^{\frac 1 2}  \\ \notag
&\leq C(\int_{0}^{t}(1+ t-s')^{-d}\|g\|^2_{L^2(\mathcal{L}^2)}ds')^{\frac 1 4}  \\ \notag
&\leq C(1+t)^{-\frac d 4}.
\end{align}
Using \eqref{tau1}, \eqref{ineq3}, Remark \ref{rema} and Theorem \ref{th1}, we have
\begin{align}\label{ineq4}
&\|\int_{0}^{t}e^{(t-s')A}\mathcal{F}_x(0,{\rm div}(i(\rho)\tau))^T ds'\|_{L^2} \\ \notag
&\leq C\int_{0}^{t}(\int_{|\xi|\leq r_1}e^{2(t-s')A}|\xi|^2|\mathcal{F}_x(i(\rho)\tau)|^2 d\xi)^{\frac 1 2}+(\int_{|\xi|> r_1}e^{2(t-s')A}|\xi|^2|\mathcal{F}_x(i(\rho)\tau)|^2 d\xi)^{\frac 1 2}ds'  \\ \notag
&\leq C\int_{0}^{t}(1+ t-s')^{-\frac d 4}\|\mathcal{F}_x(i(\rho)\tau)\|_{L^d}+\|\tau\nabla i(\rho)\|_{L^2}+\|\nabla\tau\|_{L^2})ds'  \\ \notag
&\leq C\int_{0}^{t}(1+ t-s')^{-\frac {d} 4}(\|\tau\|_{L^{\frac {d} {d-1}}}+\|\tau\|_{L^2}\|\nabla \rho\|_{H^{s-1}}+\|\tau\|_{L^2}^{1-\frac 1 {s}}
\|\tau\|_{H^{s}}^{\frac 1 {s}})d s'  \\ \notag
&\leq C(1+t)^{-\frac {d} 4}+C\int_{0}^{t}(1+ t-s')^{-\frac {d} 4}(\|\tau\|^{\frac 2 {d}}_{L^2}\|\tau\|^{1-\frac 2 {d}}_{L^1}
+\|\tau\|_{L^2}^{1-\frac 1 {s}}\|\nabla_R g\|^{\frac 1 {s}}_{H^{s}(\mathcal{L}^2)})ds'  \\ \notag
&\leq C(1+t)^{-\frac d 4}+C\int_{0}^{t}(1+ t-s')^{-\frac d 4}(\|\tau\|^{\frac 2 d}_{L^{2}}e^{C\sqrt{s'}}+\|\tau\|^{1-\frac 1 s}_{L^{2}}
\|\nabla_R g\|^{\frac 1 s}_{H^{s}(\mathcal{L}^2)})ds'  \\ \notag
&\leq C(1+t)^{-\frac d 4}+C\int_{0}^{t}(1+ t-s')^{-\frac d 4}e^{-Cs'}(\|\nabla_R g\|^{\frac 1 d}_{L^{2}(\mathcal{L}^2)}+\|\nabla_R g\|^{\frac 1 2
+\frac 1 {2s}}_{H^{s}(\mathcal{L}^2)})ds'  \\ \notag
&\leq C(1+t)^{-\frac d 4}.
\end{align}
It follows from \eqref{eq5}-\eqref{ineq4} that
\begin{align}\label{ineq5}
\|U\|_{L^2}\leq C(1+t)^{-\frac d 4}+C\epsilon^{\frac 1 2} (\int_{0}^{t}(1+ t-s')^{-\frac d 2}\|U\|^2_{H^{s-1-\delta}}ds')^{\frac 1 2}.
\end{align}
Multiplying $|\xi|^{s-1-\delta}$ to \eqref{eq5}, using Lemma \ref{Lemma5}, we obtain the higher order derivative estimates:
\begin{align}\label{ineq6}
\|e^{tA}|\xi|^{s-1-\delta}\hat{U}(0,\xi)\|_{L^2}
&\leq (\int_{|\xi|\leq r_1}|\xi|^{2(s-1-\delta)}e^{-2\beta_{j} t|\xi|^2} d\xi)^{\frac 1 2}\|\widehat{U_0}\|_{L^\infty}+Ce^{-\beta_2 t}\|U_0\|_{H^s}  \\ \notag
&\leq C(1+t)^{-\kappa}(\|U_0\|_{L^1}+\|U_0\|_{H^s}),
\end{align}
where $2\kappa=s-1-\delta+\frac d 2$. Using Theorem \ref{th1}, then we have
\begin{align}\label{ineq7}
&\|\int_{0}^{t}e^{(t-s')A}|\xi|^{s-1-\delta}\widehat{F(U)}ds'\|_{L^2}  \\ \notag
&\leq C\int_{0}^{t}(1+ t-s')^{-\kappa}(\|F(U)\|_{L^1}+\|F(U)\|_{H^{s-1-\delta}})ds'  \\ \notag
&\leq C\int_{0}^{t}(1+ t-s')^{-\kappa}(\|U\|_{L^2}\|\nabla U\|_{H^1}+\|U\|_{H^{s-1-\delta}}\|\nabla (\rho,u)^{T}\|_{H^{s-1-\delta}\times H^{s-\delta}})ds'  \\ \notag
&\leq C\epsilon^{\frac 1 2} (\int_{0}^{t}(1+ t-s')^{-2\kappa}\|U\|^2_{H^{s-1-\delta}}ds')^{\frac 1 2}.
\end{align}
Using Lemma \ref{Lemma5}, Remark \ref{rema} and Theorem \ref{th1}, we get
\begin{align}\label{ineq8}
&\|\int_{0}^{t}e^{(t-s')A}|\xi|^{s-1-\delta}\mathcal{F}_x(0,\tau\nabla i(\rho) )^T ds'\|_{L^2} \\ \notag
&\leq C\int_{0}^{t}(1+ t-s')^{-\kappa}(\|\tau\nabla i(\rho)\|_{L^1}+\|\tau\nabla i(\rho)\|_{H^{s-1-\delta}})ds'  \\ \notag
&\leq C\int_{0}^{t}(1+ t-s')^{-\kappa}(\|\tau\|_{L^2}\|\nabla \rho\|_{L^2}+\|\tau\|_{H^{s-1-\delta}}\|\nabla \rho\|_{H^{s-1-\delta}})ds'  \\ \notag
&\leq C(\int_{0}^{t}(1+ t-s')^{-2\kappa}\|g\|_{H^{s-1-\delta}(\mathcal{L}^2)}\|\nabla_R g\|_{H^{s-1-\delta}(\mathcal{L}^2)}ds')^{\frac 1 2}  \\ \notag
&\leq C(\int_{0}^{t}(1+ t-s')^{-4\kappa}\|g\|^2_{H^{s-1-\delta}(\mathcal{L}^2)}ds')^{\frac 1 4}  \\ \notag
&\leq C(\int_{0}^{t}(1+ t-s')^{-4\kappa}(\|g\|^2_{L^2}+\|g\|^{2\frac {1+\delta} s}_{L^2}\|g\|^{2-2\frac {1+\delta} s}_{H^{s}(\mathcal{L}^2)})ds')^{\frac 1 4}  \\ \notag
&\leq C(1+t)^{-\kappa}.
\end{align}
Using \eqref{tau1}, Remark \ref{rema} and Theorem \ref{th1}, we have
\begin{align}\label{ineq9}
&\|\int_{0}^{t}e^{(t-s')A}|\xi|^{s-1-\delta}\mathcal{F}_x(0,{\rm div}(i(\rho)\tau))^T ds'\|_{L^2}  \\ \notag
&\leq C\int_{0}^{t}(\int_{|\xi|\leq r_1}e^{2(t-s')A}|\xi|^{2s-2\delta}|\widehat{(i(\rho)\tau)}|^2 d\xi)^{\frac 1 2}+(\int_{|\xi|> r_1}e^{2(t-s')A}|\xi|^{2s-2\delta}|\widehat{i(\rho)\tau}|^2 d\xi)^{\frac 1 2}ds'  \\ \notag
&\leq C\int_{0}^{t}(1+ t-s')^{-\kappa}(\|\widehat{(i(\rho)\tau)}\|_{L^d}+\|(i(\rho)-1)\tau\|_{H^{s-\delta}}+\|\tau\|_{H^{s-\delta}})ds'  \\ \notag
&\leq C\int_{0}^{t}(1+ t-s')^{-\kappa}(\|\tau\|_{L^{\frac d {d-1}}}+\|\tau\|_{H^{s-\delta}})ds'  \\ \notag
&\leq C\int_{0}^{t}(1+ t-s')^{-\kappa}(\|\tau\|^{\frac 2 d}_{L^{2}}\|\tau\|^{1-\frac 2 d}_{L^{1}}+\|\tau\|_{L^{2}}+\|\tau\|^{\frac \delta s}_{L^{2}}\|\nabla_R g\|^{1-\frac \delta s}_{H^{s}(\mathcal{L}^2)})ds'  \\ \notag
&\leq C\int_{0}^{t}(1+ t-s')^{-\kappa}e^{-Cs'}(\|\nabla_R g\|^{\frac 1 d}_{L^{2}(\mathcal{L}^2)}+\|\nabla_R g\|^{\frac 1 2}_{L^{2}(\mathcal{L}^2)}+\|\nabla_R g\|^{1-\frac \delta {2s}}_{H^{s}(\mathcal{L}^2)})ds'  \\ \notag
&\leq C(1+t)^{-\kappa}.
\end{align}
It follows from \eqref{ineq6}-\eqref{ineq9} that
\begin{align}\label{ineq10}
\||\xi|^{s-1-\delta}\widehat{U}\|_{L^2}\leq C(1+t)^{-\kappa}+C\epsilon^{\frac 1 2} (\int_{0}^{t}(1+ t-s')^{-2\kappa}\|U\|^2_{H^{s-1-\delta}}ds')^{\frac 1 2}.
\end{align}
Define $M(t)=\sup_{0\leq s'\leq t} (1+s')^{\frac d 4} \|U(s')\|_{H^{s-1-\delta}}$. According to \eqref{ineq5} and \eqref{ineq10}, if $d\geq 3$, then we have
$$M(t)\leq C+C\epsilon^{\frac 1 2}M(t).$$
Finally, we get $M(t)\leq C$, which means that we obtain optimal decay rate \eqref{decay} of $\rho$ and $u$. We thus complete the proof the Theorem \ref{th2}.
\hfill$\Box$
\begin{rema}
	In Theorem \ref{th2}, we only obtain optimal decay rate with $d\geq3$. The decay rate for $d=2$ is an interesting problem. Global existence of the strong
	solutions for the FENE dumbbell model with large data is a challenging problem. However, the
	technique in this paper fails to deal with the problems. We are going to study about these problems in the future.
\end{rema}

\smallskip
\noindent\textbf{Acknowledgments} This work was
partially supported by the National Natural Science Foundation of China (No.12171493 and No.11671407), the Macao Science and Technology Development Fund (No. 0091/2018/A3), Guangdong Province of China Special Support Program (No. 8-2015),
the key project of the Natural Science Foundation of Guangdong province (No. 2016A030311004), and National Key R$\&$D Program of China (No. 2021YFA1002100).

\noindent\textbf{Data Availability.}
The data that support the findings of this study are available on citation. The data that support the findings of this study are also available
from the corresponding author upon reasonable request.

%The authors thank the referee for valuable comments and suggestions.

\phantomsection
\addcontentsline{toc}{section}{\refname}
%添加参考文献到书签，宏包 hyperref
\bibliographystyle{abbrv} %plain ,%alpha, %abbrv
\bibliography{Feneref}

\end{document}